\documentclass[final,leqno]{siamltex}

\usepackage{amsmath}
\usepackage{amssymb}
\usepackage{a4}
\usepackage{mathrsfs} \numberwithin{equation}{section}
\usepackage{latexsym} % needed in Karlin so that \Box works

\usepackage{psfrag}
\usepackage{epsfig}
\usepackage{enumitem}
\usepackage{relsize}
\usepackage{cite}
\def\version{1.4.2015}\def\users{}  % 
\usepackage{mathrsfs}
\usepackage{color}
\usepackage{psfrag}     % for description of figures - needs going thru ps!
\usepackage{ifthen}
\ifthenelse{\equal{\users}{final-layout}}{}{
\usepackage{fancyhdr}
\pagestyle{fancy}
\headheight=28pt\headwidth=16.5cm
\definecolor{gray}{gray}{0.5}
\rhead{\color{green}Plasticity with damage at small strains\\
T.Roub\'\i\v cek \& J.Valdman}
\chead{}
\lhead{\color{green}Version \version, file: \jobname.tex\\
compiled:
\number\day.\number\month.\number\year\ at
\the\hour:\ifnum\minute<10 0\fi\the\minute\ h }
}

\usepackage[notcite,notref,color]{showkeys}
\definecolor{labelkey}{rgb}{1.,.2,0.}

\numberwithin{equation}{section}

\newcount\hour \newcount\minute
\hour=\time
\divide \hour by 60
\minute=\time
\loop \ifnum \minute > 59 \advance \minute by -60 \repeat

\usepackage[normalem]{ulem}
\definecolor{brown}{rgb}{0.5,0,0}
\ifthenelse{\equal{\users}{final-layout}}{

    \newcommand{\DELETE}[1]{}
    
    \newcommand{\COMMENT}[1]{}
    
    \newcommand{\REM}[1]{\marginpar{\bfseries\tiny{\color{blue}}}}
    \newcommand{\TODO}[1]{} 
}
{

 \newcommand{\DELETE}[1]{{\color{brown}\sout{#1}\color{black}}}
 
 \newcommand{\COMMENT}[1]{{\color{red}\uuline{#1}\color{black}}}
 
 \newcommand{\REM}[1]{\marginpar{\bfseries\tiny{\color{blue}#1}}}
 \newcommand{\TODO}[1]{{$^{\color{blue}{TODO:}}$\footnote{\color{blue}{#1}}}}  
}

\newtheorem{remark}[theorem]{Remark}

\newcommand\DT[1]{\mathchoice
                 {{\buildrel{\hspace*{.1em}\text{\LARGE.}}\over{#1}}}
                 {{\buildrel{\hspace*{.1em}\text{\Large.}}\over{#1}}}
                 {{\buildrel{\hspace*{.1em}\text{\large.}}\over{#1}}}
                 {{\buildrel{\hspace*{.1em}\text{\large.}}\over{#1}}}}
\newcommand\DDT[1]{\mathchoice
   {{\buildrel{\hspace*{.1em}\text{\LARGE.\hspace*{-.1em}.}}\over{#1}}}
   {{\buildrel{\hspace*{.1em}\text{\Large.\hspace*{-.1em}.}}\over{#1}}}
   {{\buildrel{\hspace*{.1em}\text{\large.\hspace*{-.1em}.}}\over{#1}}}
   {{\buildrel{\hspace*{.1em}\text{\large.\hspace*{-.1em}.}}\over{#1}}}}

\newcommand{\GD}{\mathchoice
                  {\Gamma_{\hspace*{-.15em}\mbox{\tiny\rm D}}}
                  {\Gamma_{\hspace*{-.15em}\mbox{\tiny\rm D}}}
                  {\Gamma_{\hspace*{-.1em}\mbox{\tiny\rm D}}}
                  {\Gamma_{\hspace*{-.05em}\mbox{\tiny\rm D}}}}
\newcommand{\GN}{\Gamma_{\hspace*{-.15em}\mbox{\tiny\rm N}}}
\newcommand{\Sdir}{\Sigma_{\mbox{\tiny\rm D}}}
\newcommand{\Snew}{\Sigma_{\mbox{\tiny\rm N}}}
\newcommand{\sY}{\sigma_{\mbox{\tiny\rm Y}}^{}}
\newcommand{\scrE}{\mathscr E}
\newcommand{\scrR}{\mathscr R}
\newcommand{\barQ}{{\,\overline{\!Q\!}\,}}
\newcommand{\barOmega}{{\,\overline{\!\Omega\!}\,}}

\newcommand\eq{\eqref}
\newcommand\ti{\times}
\newcommand\eps{\varepsilon}
\newcommand{\R}{{\mathbb R}} 
 
\newcommand\uD{u_{\mbox{\tiny\rm D}}}
\newcommand\DTuD{\DT u_{\mbox{\tiny\rm D}}}
\newcommand\wD{w_{\mbox{\tiny\rm D}}}
\newcommand\hD{\mathfrak{h}_{\mbox{\tiny\rm D}}}

\newcommand\baruDtau{\baru_{\mbox{\tiny\rm D},\tau}}
\newcommand{\weak}{\rightharpoonup}

\newcommand{\wt}[1]{\mathchoice
     {\text{\small$\widetilde{\text{\normalsize$#1$}}\hspace*{.03em}$}}
                    {\text{\small$\widetilde{\text{\normalsize$#1$}}$}}
                    {\widetilde{#1\hspace*{-.02em}}\hspace*{.03em}}
                    {\tilde{#1}}}
\newcommand{\Indic}{\boldsymbol\delta}
\renewcommand{\d}{\mathrm{d}}
\newcommand{\dd}{\,\mathrm{d}}
\newcommand{\frM}{f}
\newcommand\ootimes{\odot}
\newcommand{\nablaS}{\nabla_{\scriptscriptstyle\textrm{\hspace*{-.1em}S}}^{}}
\newcommand{\divS}{\mathchoice
                   {\mathrm{div}_{\scriptscriptstyle\textrm{\hspace*{-.1em}S}}^{}}
                   {\mathrm{div}_{\scriptscriptstyle\textrm{\hspace*{-.2em}S}}^{}}
                   {\mathrm{div}_{\scriptscriptstyle\textrm{\hspace*{-.3em}S}}^{}}
                   {\mathrm{div}_{\scriptscriptstyle\textrm{\hspace*{-.5em}S}}^{}}}
\newcommand{\PS}{\mathrm{\mathbb P}_{\scriptscriptstyle\textrm{\hspace*{-.2em}S}}^{}}
\newcommand{\Vdots}{\mathchoice{\,\vdots\,}
        {\,\begin{minipage}[c]{.1em}\vspace*{-.3em}$^{\vdots}$\end{minipage}\,}
        {\,\tiny\vdots\,}{\,\tiny\vdots\,}}
\newcommand{\Item}[2]{\parbox[t]{.055\textwidth}{#1}\hfill%
      \parbox[t]{.945\textwidth}{#2}\vspace*{.8mm}} 

\DeclareMathOperator*{\argmin}{argmin}
\DeclareMathOperator*{\tr}{tr}

\newcommand\uDtauhk{u_{\mbox{\tiny\rm D},\tau h}^k}

\newcommand{\baru}{{\hspace*{.05em}\overline{\hspace*{-.05em}u\hspace*{-.05em}}\hspace*{.05em}}}
\newcommand{\bare}{{\hspace*{.05em}\overline{\hspace*{-.05em}e}}}
\newcommand{\barpi}{{\hspace*{.05em}\overline{\hspace*{-.05em}\pi\hspace*{-.05em}}\hspace*{.05em}}}
\newcommand{\barzeta}{{\hspace*{.05em}\overline{\hspace*{-.05em}\zeta}}}
\newcommand{\barxi}{{\hspace*{.05em}\overline{\hspace*{-.05em}\xi}}}

\DeclareMathOperator*{\dev}{dev}

\title{Perfect plasticity with damage  and healing at small strains, its modelling, analysis, and computer implementation}

\author{Tom\'a\v s Roub\'\i\v cek\footnotemark[1]\ \footnotemark[2]\ \footnotemark[3]\and 
Jan Valdman\footnotemark[3] \footnotemark[4]}

\begin{document}
\maketitle

\renewcommand{\thefootnote}{\fnsymbol{footnote}}

\footnotetext[1]{Mathematical Institute, Charles University,
%MFF UK, 
Sokolovsk\'a 83, CZ-186~75~Praha~8, Czech Rep.}
\footnotetext[2]{Institute of Thermomechanics,
Czech Acad.\ Sci.,
% of the CAS, 
Dolej\v skova 5, 182 00 Praha 8, Czech Rep.}
\footnotetext[3]{Institute of Information Theory and Automation, Czech Academy of Sciences,
Pod vod\'{a}renskou v\v{e}\v{z}\'{\i}~4,
CZ-18208~Praha~8, Czech Republic.}
\footnotetext[4]{Institute of Mathematics and Biomathematics, 
Faculty of Science, University of South Bohemia, Brani\v sovsk\' a~31, 
CZ-370 05 \v{C}esk\'{e} Bud\v{e}jovice, Czech Republic.}

\begin{abstract}
The quasistatic, Prandtl-Reuss perfect plasticity at small strains is combined with a gradient, reversible (i.e.\ admitting healing) damage which influences both the elastic moduli and the yield stress. Existence of weak solutions of the resulted system of variational inequalities is proved by a suitable fractional-step discretisation in time with guaranteed numericalstability and convergence. After finite-element approximation, this scheme is computationally implemented and illustrative 2-dimensional simulations are performed. The model allows e.g.\ for application in geophysical modelling of re-occurring rupture of lithospheric faults. Resulted incremental problems are solved in MATLAB by quasi-Newton method to resolve elastoplasticity component of the solution while  damage component is obtained by solution of a quadratic programming problem.  
\end{abstract}



\pagestyle{myheadings} \thispagestyle{plain}
\markboth{T.Roub\'\i\v cek \& J.Valdman}
{Perfect plasticity with damage and healing.}

\section{Introduction}
%        ~~~~~~~~~~~~

There is a vast amount of literature about plasticity and about 
damage separately, both in mathematics and in civil or mechanical engineering.
Much less literature addresses various combination of plasticity and 
damage, cf.\ e.g.\ \cite{AlMaVi14GDMC,AlMaVi14GDMC+,Cris??GSQE,CriLaz??VAQE,JirBaz02IAS,Kach90ICDM,SoHaSh06AACP}.
%\COMMENT{check Jirasek-Bazant' book??}
In engineering, this is usually called ductile damage, 
cf.\ e.g.\ \cite{GraJir06PMND,Kraj89DM,Lema96CDM,LemDes05EDM,Maug92TPF}.
Also a lot of geophysical models combine reversible damage (called rather 
ageing) with some sort of plasticity (often modelled as not entirely 
independent of damage, however), cf.\ e.g.\ \cite{LyBZAg97DDFF}.

The goal of this article is to devise a model that would allow for\\
\Item{$\bullet$}{modelling of thin plastic shear bands surrounded by wider damage zones 
(as typically occurs in geophysical modeling of lithospheric faults with very 
narrow core) with 
possible healing of damage (as considered in geophysical modeling to allow 
re-occurring damaging), and simultaneously}
\Item{$\bullet$}{rigorous proof of existence of weak solutions of the resulted system of 
variational inequalities proved by a suitable fractional-step discretisation 
in time with guaranteed numerical stability and convergence, and}
\Item{$\bullet$}{efficient numerical implementation of the time-discrete model.}
%\begin{itemize}\vspace*{-.0em}\item
%modelling of thin plastic shear bands surrounded by wider damage zones 
%(as typically occurs in geophysical modeling of lithospheric faults with very 
%narrow core) with possible healing of damage (as considered in geophysical 
%modeling to allow re-occurring damaging), and simultaneously 
%\vspace*{-.0em}\item
%rigorous proof of existence of weak solutions of the resulted system of 
%variational inequalities proved by a suitable fractional-step discretisation 
%in time with guaranteed numerical stability and convergence, and
%\vspace*{-.0em}\item
%efficient numerical implementation of the time-discrete model.\end{itemize}
We depart from the standard linearized, associative, rate-independent 
plasticity at small strain as presented e.g.\ in \cite{HanRed99PMTN}.
Simultaneously, we use also a rather standard scalar (i.e.\ isotropic)
damage as introduced by L.M.\,Kachanov in late 60ieth and 
presented e.g.\ in \cite{Frem02NST}, considered here 
however as rate dependent and reversible in the sense that 
a possible healing is allowed. To avoid serious mathematical and 
computation difficulties, we have in mind primarily 
%consider only 
an incomplete damage through a higher-order damage-independent 
term, although the standard elastic tensor can allow for a complete 
damage, cf.\ $\mathbb H$ and $\mathbb C=\mathbb C(\zeta)$ below. 
%which still have reasonable applications.
% and a good sense in particular in the context of healing 
%(as completely damaged ...???
An important aspect of the model
is that not only the conservative part but also the dissipative
part is subjected to damage, i.e.\ not only the elastic moduli but
also the yield stress will be considered as damageable. This relatively
simple and lucid mechanism will however lead to a possibly very complex
response of the model. 
%\INSERT{.... SOMEWHERE A REMARK: more general 
%dissipative mechanism $\scrR(\zeta,\pi;\DT\pi,\DT\zeta)$: damage 
%threshold dependent on a ``cumulative'' 
%plasticity as in \cite{AlMaVi14GDMC} --- i.e. \ here $a$ dependent on
%$\int_0^t|\DT\pi|\d t$ (+maybe also in \cite{Cris??GSQE}- HOW TO WRITE IT WITHOUT HARDENING?? }

To make the model accessible to analysis, we work within the setting of 
small strains, and we also take into account surface-energy effects by 
including in the free energy a term dependent on the gradient of the total 
strain. This is also known as a concept of so-called second-grade nonsimple 
materials, cf.\ e.g.\ \cite{Podi02CISM,Silh88PTNB}, alternatively also 
referred as the concept of hyper- or couple-stresses 
\cite{PoGiVi10HHCS,Toup62EMCS}; for reasons we use it here 
cf.\ Remark~\ref{rem-simple} below. 
%This gradient theory for small strains together with 
%the standard exchange-energy term for magnetization enables us to treat 
%{\it nonconvex} free energies $\varphi(\cdot,\theta)$, cf.\ \eqref{freen} with 
%\eqref{phi-growth}, a feature that is essential to describe phase 
%transformation.\COMMENT{ STILL TO MODIFY}

In view of applications we have in mind, we suppress any hardening effects
and thus we consider the \emph{Prandtl-Reuss} elastic/\emph{perfectly plastic} 
model; in fact, considering kinematic or isotropic hardening would make
a lot of aspects even much easier. A plastic yield stress dependent on damage 
is in some variants used in the Cam-Clay model, cf.\ e.g.\ 
\cite{DaDeSo11QECC,LICa02CamClay,Whee03CamClay2}, or in the Perzyna model with 
damage, cf.\ \cite{SoHaSh06AACP},
and also in \cite{AlMaVi14GDMC,AlMaVi14GDMC+,Cris??GSQE,CriLaz??VAQE}. 
%\COMMENT{perfect plasticity with damage-dependent yield stress is in \cite{Cris??GSQE} with unidirectional damage and, in 
%\cite{CriLaz??VAQE} vanishing viscosity in damage in simple material!!}
Let us also point out that damage with healing 
without plasticity (as sometimes considered in mathematical 
literature) would have only very limitted application because
damaged material typically can undergo substantial deformation 
and the healing should not be performed towards the original 
configuration.  

We confine on the isothermal variant of the model. In contrast to 
\cite{RoSoVo13MBMF}, we consider rate-independent plasticity
without any gradient, so that concentration of plastic and 
total strains and development of sharp shear bands is possible.
Also, related to this concentration, both plastification and damage 
are driven by the elastic stress (which is still well controlled) rather
than the total strain (which may concentrate); for 
plasticity itself, see also \cite{Roub13TPP}.

The presented model has potential application in geophysical modelling of 
re-occurring rupture of lithospheric faults or of nucleation of new faults. 
A narrow so-called core of the fault can be modelled by the perfect plasticity 
while and a relatively wide damage zone around it can arise by the 
gradient-damage model. After a combination with inertial effects (and possibly 
a visco-elastic rheology e.g.\ of Jeffreys type), this model involves seismic 
waves and can serve for earthquake simulations where these waves are emitted 
during fast rupture, cf.\ Remarks~\ref{rem-dynam} and \ref{rem-nonlin} below 
for some modifications of the presented model towards these applications.
Another possible modification, going beyond the scope of this paper however, 
might use the structure of the stored energy similar to what is used in a 
phenomenological models for polycrystalline shape-memory alloys where our 
damage variable is in a position of temperature and plastic strain is a 
transformation strain subjected to some additional constraints, see 
e.g.\ \cite[Example 5.15]{GKNT10ANAT}.

The plan of the paper is as follows: In Section~\ref{sec-model} we formulate 
the model and cast a suitable definition of the weak solution, and pronounce 
a basic existence result which is proved later in Sections~\ref{sec-disc}
by a constructive time discretisation method. A further finite-element 
discretisation is then outlined.
%After a further finite-element discretisation outlined in 
%Section~\ref{sec-comp}, 
This allows for computer implementation of the model presented in 
Section~\ref{sec-comp}, whose efficiency and some physical aspects 
eventually demonstrated on in Section~\ref{sec-simul} an illustrative 
example with geophysical motivation.

\section{The model, its weak formulation, and existence result}\label{sec-model}
%        ~~~~~~~~~~~~~~~~~~~~~~~~~~~~~~~~~~~~~~~~~~~~~~~~~~~~~

Hereafter, we suppose that the damageable elasto-plastic
body occupies a bounded smooth domain $\Omega\subset\R^d$, $d=2$ or $3$.
We  denote by $\vec{n}$ the outward unit normal to $\partial \Omega$.
 We further suppose that the boundary of $\Omega$ splits as
\[
\partial \Omega :=\Gamma= \GD\cup \GN\,,
\]
with $\GD$ and $\GN $ open subsets in the relative topology of
$\partial\Omega$, disjoint one from each other, each of them
with a smooth ($(d{-}1)$-dimensional) boundary, and covering 
$\partial\Omega$ up to $(d{-}1)$-dimensional
zero measure. Considering $T>0$ a fixed time horizon, we  set
\begin{displaymath}
Q:=(0,T){\times}\Omega, \qquad \Sigma:= (0,T){\times}\Gamma,
\qquad \Sdir\!:= (0,T){\times}\GD, \qquad \Snew\!:= (0,T){\times}\GN.
\end{displaymath}
Further, $\R_\mathrm{sym}^{d\ti d}$ and $\R_\mathrm{dev}^{d\ti d}$ will denote 
the set of symmetric or symmetric trace-free (=\,deviatoric) 
$(d{\ti}d)$-matrices, respectively. For readers' convenience, let us summarize 
the basic notation used in what follows:

%\vspace{.7em}

\hspace*{-1.6em}\fbox{
\begin{minipage}[t]{0.46\linewidth}
\small

$d=2,3$ dimension of the problem,

$\R_\mathrm{dev}^{d\ti d}:=\{A\in \R;\ \mathrm{tr}\,A=0\}$, 

$u:Q\to\R^d$ displacement,

%$\theta:\Omega{\setminus}\GC\to(0,+\infty)$ absolute temperature,

$\pi:Q\to\R_\mathrm{dev}^{d\ti d}$ plastic strain,

$\zeta:Q\to[0,1]$ damage variable,

%$\sigma$  stress tensor,

$a:\R\to\R^+$ damage-dissipation potential,

$b:[0,1]\to\R$ stored energy of damage,

$e_\mathrm{el}$ elastic strain, $\ e_\mathrm{el}=e(u){-}\pi$,

$e=e(u)=\frac12\nabla u^\top\!+\frac12\nabla u$ 
\\\hspace*{6em}total small-strain tensor,

$\mathbb C:[0,1]\to\R^{3^4}$ elasticity tensor\\\hspace*{6em}dependent on $\zeta$,

\end{minipage}\ \
\begin{minipage}[t]{0.48\linewidth}\small

$\mathfrak{h}$ hyperstress (3rd-order) tensor

$\mathbb H$ a (small) hyperelasticity tensor,

$S=\sY(\cdot)B_1:[0,1]\rightrightarrows\R_\mathrm{dev}^{d\ti d}$,\\\hspace*{6em}with $B_1$ the unit ball in $\R_\mathrm{dev}^{d\ti d}$,

$\sY:[0,1]\to\R^+$ plastic yield stress\\\hspace*{6em}dependent on $\zeta$,

$g:Q\to\R^d$  applied bulk force,

$\wD:\Sdir\to\R^d$  prescribed time-dependent\\\hspace*{6em}boundary displacement,

$\frM:\Snew\to\R^d$  applied traction force,

$\kappa>0$ scale coefficient\\\hspace*{6em}of the gradient of damage.

\end{minipage}\medskip
}

\vspace{-.0em}

%\INSERT{NEZJEDNODUSIME $\bbD=\tau_{\rm KV}\mathbb{C}$ A $\bbG=\tau_{\rm KV}\bbH$, 
%$\tau_{\rm KV}>0$ RELAXACNI CAS k-v MODELU?? ALE PAK NEMUZEM MIT ZVLAST JAKO
%MOZNOST $\bbG=0$ A $\bbH>0$}
\begin{center}
{\small\sl Table\,1.\ }
%\begin{minipage}[t]{.7\textwidth}\baselineskip=8pt
{\small
Summary of the basic notation used thorough the paper. 
}
%\end{minipage}
\end{center}

\noindent The {\it state} is formed by the triple $q:=(u,\pi,\zeta)$.
Considering still a (small but fixed) regularizing parameter $\eps>0$,
the governing equation/inclusions read as:
\begin{subequations}\label{plast-dam}
\begin{align}\label{plast-dam1}
&
%\varrho\DDT{u}-
\mathrm{div}\big(
%\mathbb D e(\DT{u})
%%+\mathbb C (\zeta)(e(u){-}\pi)+\eps(1{+}|e(u){-}\pi|^2)^{p/2-1})(e(u){-}\pi)
%+
\mathbb C (\zeta)e_\mathrm{el}
%+\eps|e_\mathrm{el}|^{p-2}e_\mathrm{el}
-\mathrm{div}\,\mathfrak{h}\big)+g=0\ \ \ \ \ \ \ \ \ \ 
&&\!\!\!\!\!\!\!\text{\sf(momentum equilibrium)}
\\[-.2em]&\nonumber\hspace*{4em}
\text{ with }\ \
\mathfrak{h}=\mathbb H\nabla e_\mathrm{el}
\ \text{ and }\ 
e_\mathrm{el}=e(u){-}\pi,
%&&\!\!\!\!\!\!\!\!\!\!\!\!\text{\sf(momentum equlibrium)}
\\[.2em]\label{plast-dam12}
&
\partial\delta_{S(\zeta)}^*(\DT{\pi})
%+\mathbb H\pi
%+\mathrm{dev}\,\mathbb C (\zeta)\pi\ni\mathrm{dev}\,\mathbb C (\zeta) e(u)
\ni\mathrm{dev}\big(\mathbb C (\zeta)e_\mathrm{el}
%+\eps|e_\mathrm{el}|^{p-2}e_\mathrm{el}
-\mathrm{div}\,\mathfrak{h}\big)
%\ \ \ \ \ \text{ with }\ \ e_\mathrm{el}=e(u){-}\pi
,&&\text{\sf(plastic flow rule)}
\\[-.2em]\label{plast-dam13}
&\partial a(\DT\zeta)+\frac12\mathbb C'(\zeta)
%(e(u){-}\pi):(e(u){-}\pi)
e_\mathrm{el}:e_\mathrm{el}
\\[-.2em]&\nonumber\hspace*{1em}
-\kappa
\,\mathrm{div}\big((1{+}\eps|\nabla\zeta|^{r-2})\nabla\zeta\big)
%\Delta\zeta
+N_{[0,1]}(\zeta)\ni b'(\zeta),&&\text{\sf(damage flow rule)}
\end{align}\end{subequations}
with $\delta_{S}$ the indicator function to $S$ and $\delta_{S}^*$
its convex conjugate. Here, $[\mathbb C (\zeta)e]_{ij}$ and 
$[\mathbb H\nabla e]_{ijk}$ 
mean $\sum_{k,l=1}^d\mathbb C_{ijkl}(\zeta)e_{kl}$ and 
$\sum_{m,n=1}^d\mathbb H_{ijmn}\frac{\partial}{\partial x_m}e_{in}$, respectively.
 
We employed two regularizing terms with a regularizing tensor $\mathbb H$
and a regularizing parameter 
$\eps>0$ with  an exponent to be assumed suitably big, namely 
%$p>4$ and 
$r>d$. 
This regularization facilitates analytical well-posedness of 
the problem and, because the gradient-damage term
%s are 
degenerates at 
%$e_\mathrm{el}=0$ and 
$\nabla\zeta=0$, its influence is 
presumably small if $\eps$ is small and 
%$e_\mathrm{el}$ and 
$\nabla\zeta$ not too large. Moreover, 
%$\eps$ 
$\mathbb H$ in 
\eqref{plast-dam1} prevents a complete damage
%; of course, 
at least when we assume $\mathbb C(\zeta)$ positive semidefinite.
%{\tiny .........$\frac12\sum_{i,k,l,m,n=1}^d\mathbb H_{klmn}
%%G_{ikl}G_{imn}
%\frac{\partial^2u_i}{\partial x_k\partial x_l}\frac{\partial^2u_i}{\partial x_m%\partial x_n}$ For example, 
%$\mathbb H_{klmn}=\delta_{kl}\delta_{mn}$ yields 
%$\sum_{i,k,l,m,n=1}^d\mathbb H_{klmn}\frac{\partial^2u_i}{\partial x_k\partial x_l}
%\frac{\partial^2u_i}{\partial x_m\partial x_n}=\sum_{i=1}^d|\Delta u_i|^2$ 
%or $\mathbb H_{klmn}=\delta_{km}\delta_{ln}$ yields
%$\sum_{i,k,l,m,n=1}^d\mathbb H_{klmn}\frac{\partial^2u_i}{\partial x_k\partial x_l}
%\frac{\partial^2u_i}{\partial x_m\partial x_n}=\sum_{i=1}^d|\nabla^2u_i|^2$; 
%here $\delta$ denotes Kronecker's symbol.}
Actually, \eqref{plast-dam12} represents rather the thermodynamical-force
balance governing damage evolution while the corresponding flow rule is written 
rather in the (equivalent) form
\[
\DT{\pi}
%+\mathbb H\pi
%+\mathrm{dev}\,\mathbb C (\zeta)\pi\ni\mathrm{dev}\,\mathbb C (\zeta) e(u)
\in N_{S(\zeta)}\Big(\mathrm{dev}\big(\mathbb C (\zeta)e_\mathrm{el}
%+\eps|e_\mathrm{el}|^{p-2}e_\mathrm{el}
-\mathrm{div}\,\mathfrak{h}\big)\Big)
\]
with $N$ the set-valued normal-cone mapping to the convex set indicated.
An analogous remark applies to \eqref{plast-dam13}.

A remarkable attribute of this model is a
damage-dependent yield-stress domain $S=S(\zeta)$. Typically, 
developing damage makes $S$ smaller and vice versa, i.e.\ 
$S(\cdot):[0,1]\rightrightarrows\R_\mathrm{dev}^{d\ti d}$ 
is nondecreasing with respect to the ordering of subsets by inclusion. 
%more damaged material has smaller $S$. 
Likewise, typically also $b(\cdot)$ and $\mathbb C(\cdot)$ are nondecreasing,
the later one with respect to the L\"owner's ordering, i.e.\ 
$\mathbb C(z_1)-\mathbb C(z_2)$ is positive semi-definite for $z_1\ge z_2$.
%Moreover, assuming $b'(1)=0$ and $\mathbb C'(0)=0$, the 
%damage variable will always range $[0,1]$ if the initial 
%condition $\zeta(0,\cdot)=\zeta_0^{}$ is within these bounds, which 
%is due to the maximum principle valid for the Laplacian in \eq{plast-dam13}. 
Rate-dependency of damage evolution prevents nonphysically too-early 
damaging/plastification and, due to the driving force 
$b'(\zeta)$, also allows simply for reverse damage evolution (a so-called 
\emph{healing})\index{healing}\index{damage!with healing} 
by using a convex function $a:\R\to\R^+$ in \eq{plast-dam13} 
%not taking $\infty$-values and coercive in the sense 
%$a(\zeta)\ge\eps|\zeta|^q$ with some $\eps>0$ and $q>1$
having naturally its minimum at $0$. The microstructural interpretation of $b$ 
is a stored energy related with microcracks/microvoids arising by damage, 
reflecting the fact that any surface in the bulk bears some extra energy.
Minimization of this energy naturally leads to a tendency for healing of 
these material defects.
%\REM{ALSO unidirectional variant:
%$b=0$, $\mathbb C '(0)=0$, $p\ge4$, 
%%no explicit constraints $0\le\zeta\le1$ allows for 
%%standard existence theory for \eq{eq6:plast-dam13}, 
%cf.\ \cite{Roub13NPDE}.}
Of course, \eq{plast-dam} is to be completed by appropriate boundary 
conditions for (\ref{plast-dam}a,c), e.g.
\begin{subequations}\label{plast-dam-BC}
\begin{align}\label{plast-dam-BC1}
&&&u=\wD&&\text{on }\GD,&&&&
\\\label{plast-dam-BC+}
&&&\big(\mathbb C (\zeta)e_\mathrm{el}
%+\eps|e_\mathrm{el}|^{p-2}e_\mathrm{el}
-\mathrm{div}\,\mathfrak{h}\big){\cdot}\vec{n}
-\divS(\mathfrak{h}\vec{n})=\frM
%\ \ \text{ and }\ \ \mathfrak{h}{:}(\vec{n}\otimes\vec{n})=0
&&\text{on }\GN,
\\&&&
%(1{+}\eps|\nabla\zeta|^{r-2})
\nabla\zeta{\cdot}\vec{n}=0
\ \ \ \ \text{ and }\ \ \ \ \mathfrak{h}{:}(\vec{n}\otimes\vec{n})=0
&&\text{on }\Gamma
\end{align}\end{subequations}
with $\vec{n}$ denoting the unit outward normal to $\Omega$. Moreover, 
$\divS$ is the surface-divergence operator, which may be introduced as follows 
\cite{GurMur74CTEM}: given a vector field $v:\Gamma\rightarrow\R^d$, we extend 
it to a neighborhood of $\Gamma$, and we let its surface gradient (valued
in $\R^{d\times d}$) be defined 
as $\nablaS v=\PS\nabla v$, where $\PS=\mathbb I-\vec{n}\otimes\vec{n}$ is the 
projector on the tangent space of $\Gamma$; we then let the surface divergence 
of $v$ be the scalar field $\divS v=\PS:\nablaS v={\rm tr}(\PS\nabla v\PS)$. 
Given a tensor field $\mathbb A:\Gamma\rightarrow\R^{d\times d}$, we let 
$\divS\mathbb A:\Gamma\rightarrow\R^d$ be the unique vector field such 
that $\divS(\mathbb A^Ta)=a{\cdot}\divS\mathbb A$ for all constant vector 
fields $a:\Gamma\rightarrow\R^d$. {F}urthermore, the symbols ``$\,\cdot\,$'' 
and ``$\,:\,$'' denote a contraction between the one or 
%last 
two indices, respectively. Later, we will use also ``$\,\Vdots\,$'' for 
a contraction between three indices. Thus, componentwise, the second 
condition in \eqref{plast-dam-BC+} reads as $\sum_{j,k=1}^d\mathfrak{h}_{ijk}n_jn_k=0$.

Of course, an inhomogeneous variant of \eqref{plast-dam-BC+}
or some mixed Dirichlet/Neumann conditions in the normal/tangent
conditions could be considered with straightforward modifications 
of the following text. We will consider an initial-value problem 
for \eqref{plast-dam}--\eqref{plast-dam-BC} by asking for 
\begin{align}\label{plast-dam-IC}
u(0)=u_0,\ \ \ \ \ \pi(0)=\pi_0,\ \ \text{ and }\ \ \zeta(0)=\zeta_0.
\end{align}
In fact, as $\DT u$ does not occur in \eqref{plast-dam}, $u_0$ is 
rather formal and will essentially be determined by $\pi_0$ and 
$\zeta_0$ via \eqref{ass-IC-stable} below.

The system \eq{plast-dam} with the boundary conditions \eq{plast-dam-BC}
has, in its weak formulation, the structure of an abstract 
Biot equation (or here rather inclusion): 
\begin{align}\label{Biot}
%{\mathscr M}'\DDT{q} +
\partial_{\DT q}\scrR(q;\DT q)+
\partial\scrE(t,q)\ni0
%{\mathscr F}(t,q)
\end{align}
with suitable time-dependent stored-energy functional $\scrE$ and the 
state-dependent (pseudo)potential of dissipative forces $\scrR$. 
Equally, one can write \eq{Biot} as a generalized gradient flow 
\begin{align}
\DT q\in \partial_{\xi}\scrR^*\big(q;-\partial\scrE(t,q)\big)
\end{align}
where $\xi\mapsto\scrR^*(q;\xi)$ denotes the conjugate functional 
to $v\mapsto\scrR(q;v)$.

The perfect-plasticity model itself received considerable attention already 
a long time ago, see e.g.~in 
%\cite{EboRed04MPPP,John76ETPP,JohSco81FEMP,Maug92TPF,Rep96EFEM}
\cite{BaMiRo12QSSP,DaDeMo06QEPL,EboRed04MPPP,John76ETPP,Maug92TPF,Rep96EFEM}.
The peculiarity is that the displacement no longer lives in
the conventional Sobolev $H^1$-space but rather in the 
space of \emph{functions with bounded 
deformations} introduced by Suquet \cite{Suqu78ERSE}, defined as
\begin{align}\label{BD}
\mathrm{BD}(\barOmega;\R^d):=\big\{u\!\in\!L^1(\Omega;\R^d);\ 
e(u)\!\in\!\mathrm{Meas}(\barOmega;\R^{d\ti d}_{\mathrm{sym}})\big\},
\end{align}
where $\mathrm{Meas}(\barOmega)\cong C(\barOmega)^*$ denotes the space of 
Borel measures on the closure of $\Omega$. The other notation we will use is 
rather standard: beside the standard notation for the Lebesgue $L^p$-space we 
already used in \eqref{BD} for $p=1$, we further use $W^{k,p}$ for Sobolev space
whose $k$-th derivatives are in $L^p$-spaces, the abbreviation $H^k=W^{k,2}$, 
and $L^p(0,T;X)$ for Bochner spaces of Bochner-measurable mappings $(0,T)\to X$ 
with $X$ a Banach space. Also, $W^{k,p}(0,T;X)$ denotes the Banach space of 
mappings from $L^p(0,T;X)$ whose $k$-th distributional derivative in time is 
also in $L^p(0,T;X)$. Further, $C([0,T];X)$ and $C_\text{weak}([0,T];X)$ will 
denote the Banach space of continuous and weakly continuous mappings 
$[0,T]\to X$, respectively. Moreover, we denote by ${\rm BV}([0,T];X)$ the 
Banach space of the mappings $[0,T]\to X$ that have 
a bounded variation on $[0,T]$, and by ${\rm B}([0,T];X)$ the space of 
Bochner measurable, everywhere defined, and bounded mappings $[0,T]\to X$.
%\CHECK{By ${\rm Meas}([0,T];X)$ we denote the space of $X$-valued measures on
%$[0,T]$.}\COMMENT{SHALL WE NEED IT??}

After considering smooth time-dependent Dirichlet boundary 
conditions $\wD$ on $\Sdir$ which allows for an extension onto $Q$, let
us denote it by $\uD$, such that 
%\begin{align}\label{uD=0-on-GN}
%\left.\begin{array}{l}\big(\mathbb C (\zeta)e(\uD)
%%+\eps|e_\mathrm{el}|^{p-2}e_\mathrm{el}
%-\mathrm{div}\,\hD\big){\cdot}\vec{n}
%-\divS(\hD\vec{n})=0\ \ \ \text{ and }\\[.2em]
%\hD{:}(\vec{n}\otimes\vec{n})=0\ \ \ \text{ with }\ \
%\hD=\mathbb H\nabla e(\uD)\end{array}\right\}\ \text{\ on }\GN\end{align}
\begin{subequations}\label{uD=0-on-GN}\begin{align}
&&&&&\big(\mathbb C (\zeta)e(\uD)
%+\eps|e_\mathrm{el}|^{p-2}e_\mathrm{el}
-\mathrm{div}\,\hD\big){\cdot}\vec{n}
-\divS(\hD\vec{n})=0&&\text{on }\ \GN,&&&&
\\[.2em]&&&&&
\hD{:}(\vec{n}\otimes\vec{n})=0\ \ \ \ \text{ with }\ \
\hD=\mathbb H\nabla e(\uD)&&\text{on }\ \Gamma&&&&
\end{align}\end{subequations}
for any admissible $\zeta$,
and making a substitution of $u+\uD$ instead of
$u$ into \eqref{plast-dam}--\eqref{plast-dam-BC}, we arrive to 
the problem with time-constant (even homogeneous) Dirichlet boundary
conditions. More specifically, 
\begin{subequations}\label{subst-zero-Dirichlet}
\begin{align}&\label{subst-zero-Dirichlet1}
%g\text{ in \eqref{plast-dam1} replaces by }
%g+\mathrm{div}\big(\mathbb C(\zeta)e(\uD)
%-\mathrm{div}(\mathbb H\nabla e(\uD))\big),\\
e_\mathrm{el}\ \,\text{ in \eqref{plast-dam12} replaces by }\
e_\mathrm{el}=e(u{+}\uD){-}\pi,\ \text{ and}
\\&\label{subst-zero-Dirichlet2}
\wD\ \text{ in \eqref{plast-dam-BC1} replaces by }0.
%\\&h\text{ in \eqref{plast-dam1} replaces by }
\end{align}\end{subequations}
The state space is then the Banach space
\begin{subequations}\label{seismic+}
\begin{align}\label{seismic-U+}
&U:=\big\{(u,\pi,\zeta)\!\in\!
%H^1(\Omega;\R^d)
\mathrm{BD}(\barOmega;\R^d){\ti}
\mathrm{Meas}(\barOmega;\R^{d\ti d}_{\mathrm{dev}}){\ti}W^{1,r}(\Omega);
\\&\hspace{6em}
e(u){-}\pi\!\in\!H^1(\Omega;\R^{d\ti d}_{\mathrm{sym}}),\ \ \
u
%|_{\GD}^{}
\ootimes\vec{n}\d S\!+\pi=0\ \text{ on }\ \GD\big\},
%\\& V:=H^1(\Omega;\R^d){\ti}L^q(\Omega){\ti}.................,
%\\\label{seismic-Z+}&Z:=L^2(\Omega;\R^{d\ti d}_\mathrm{dev}),
%%\\\label{seismic-X+}&
%\ \ \ \ \ \ \;X:=L^1(\Omega;\R^{d\ti d}_\mathrm{dev}),
%%\ \ \ \ \ \ \;H:=L^2(\Omega;\R^d){\ti}L^2(\Omega),
%\\
\nonumber
\intertext{where $a\ootimes b$ means the symmetrized tensorial product
$\frac12(a\otimes b+b\otimes a)$, and the functionals governing the problem 
\eqref{Biot} leading to \eqref{plast-dam}--\eqref{plast-dam-BC} with the
substitution \eqref{subst-zero-Dirichlet} are:
%\COMMENT{TO CHECK ESPECIALLY IF $\uD\ne0$}
}\nonumber\\[-3em]
&\scrE(t,u,\pi,\zeta)
:=\left\{\begin{array}{ll}
\displaystyle{\!\!
\int_{\Omega}\,\frac12\mathbb C(\zeta)\big(e(u{+}\uD(t)){-}\pi\big):\big(e(u{+}\uD(t)){-}\pi\big)
%+\frac12\mathbb H\pi:\pi
}\hspace*{-8em}
\\[-.2em]
%\nonumber
\displaystyle{\hspace*{1em}
+\frac12\mathbb H\nabla(e(u{+}\uD(t)){-}\pi)\Vdots\nabla(e(u{+}\uD(t)){-}\pi)
}\hspace*{-8em}
\\[-.2em]
%\nonumber
\displaystyle{\hspace*{1em}-b(\zeta)-g(t){\cdot}u
+\kappa\,\Big(\frac12|\nabla\zeta|^2{+}\frac\eps r|\nabla\zeta|^r\Big)
%\big(1{+}\eps|\nabla\zeta|^2\big)^{r/2}
%+\frac\eps p|e(u){-}\pi|^p
\dd x}\hspace*{-8em}
\\[-.2em]
%\nonumber
\displaystyle{\hspace*{1em}
%}&\hspace*{5em}\\[-.2em]\nonumber\displaystyle{\hspace*{9em}
-\!\int_{\GN}\!\!\frM(t){\cdot}u\dd S}
%\ -\big\langle f_{\mathrm{ext}}(t),(u,\pi,\eps)\big\rangle}
%&\text{if }u|_{\GDir}=\uD^{}(t),
%\\
\hspace*{-2em}&
\text{if }\zeta\!\in\![0,1]\text{ a.e.\ on }\Omega,
%\text{ and}
%\\[-.3em]&\text{if $e(u){-}\pi\!\in\!H^1(\Omega;\R^{d\ti d})$,}
\\[.1em]
\qquad\quad\infty&\text{otherwise,}
\end{array}\right.
%\\[-1.5em]&
\label{seismic-E+}
\\
\label{seismic-R+}
&%\scrR(\zeta;\DT\pi)
\scrR(\zeta;%\DT u,
\DT\pi,\DT\zeta):=\int_{\barOmega}\!\big[\Indic_{S(\zeta)}^*(\DT\pi)\big](\d x)
%\dd x,
%\\&
%\qqquad\calV(\DT u,\DT\zeta):=\int_\Omega
%+\frac12\mathbb D e(\DT{u}):e(\DT{u})
+\int_\Omega\!a(\DT\zeta)
%%\frac{k_0'}2|\DT\zeta^-|^2+\frac{k_0(\zeta)}2|\DT\zeta^+|^2
\dd x,
%\label{R2-seismic+}
%\\\label{M-seismic+}&\calM(\DT{u},\DT{\zeta})=\calM(\DT{u})
%:=\int_\Omega\frac\varrho2|\DT{u}|^2\dd x,
\end{align}\end{subequations}
where $\Indic_{S(\zeta)}^*$ denotes the conjugate to the indicator function 
$\Indic_{S(\zeta)}$ to the convex set $S(\zeta)$ and where the first integral in 
\eqref{seismic-R+} is an integral of a Borel measure;
counting the assumption \eqref{ass-S} below, this measure is $\sY(\zeta)|\DT\pi|$ with 
$|\DT\pi|$ the total variation of $\DT\pi$.
%together with an appropriate bulk force, cf.\ \eq{eq6:f-damage} below.
The norm on $U$ is 
\begin{align*}
\big\|(u,\pi,\zeta)\big\|^{}_U&:=\|u\|_{L^1(\Omega;\R^d)}
+\|e(u)\|_{\mathrm{Meas}(\bar\Omega;\R^{d\times d}_\mathrm{sym})}
\\&\ \ \ \ +\|\pi\|_{\mathrm{Meas}(\Omega;\R^{d\times d}_\mathrm{dev})}
+ \|e(u){-}\pi\|_{H^1(\Omega;\R^{d\times d}_\mathrm{sym})}
+\|\zeta\|_{W^{1,r}(\Omega)}.
\end{align*}

We can now state the weak formulation of the initial-boundary-value problem 
\eqref{plast-dam}--\eqref{plast-dam-IC}. As for the plastic part, we use 
the concept of the so-called energetic solution devised by Mielke and Theil 
\cite{MieThe04RIHM}, cf.\ also \cite{Miel05ERIS,MieRou15RIPT}, based on the 
energy (in)equality and the so-called stability and further employed in the 
viscous context in \cite{Roub09RIPV} with the stability condition modified to 
a semi-stability, cf.\ \eqref{semi-stab} below. Another feature of the 
following definition is that we rely on a regularity of the damage $\zeta$ so 
that $\mathrm{div}((1{+}\eps|\nabla\zeta|^{r-2})\nabla\zeta)$
%$\Delta\zeta$ 
is in duality with $\DT\zeta$ and thus, in fact, the damage flow rule 
\eqref{plast-dam13} holds even a.e.\ $Q$. Actually, we do not need
such regularity for the definition itself because
the usual weak formulation of \eqref{plast-dam13},
which would involve (not well-controlled) $\nabla\DT\zeta$ resulted from 
usage of Green's formula, could be still treated by applying a by-part
integration in time to get rid off the term 
$((1{+}\eps|\nabla\zeta|^{r-2})\nabla\zeta)\cdot\nabla\DT\zeta$.
Rather, this regularity is essential for the energy conservation.
%\COMMENT{+SOME OTHER REASONS??}

\begin{definition}[Weak solution]\label{weak-sln}
%Given initial conditions \eqref{assum-initi} and the bulk and boundary data 
%\eqref{eq:19}, 
The triple $(u,\pi,\zeta)$ 
%find functions
with 
\begin{subequations}
\begin{align} 
&u\in\mathrm{B}([0,T];\mathrm{BD}(\barOmega;\R^d)),
\\
&\pi\in\mathrm{B}([0,T];\mathrm{Meas}(\barOmega;\R^{d\ti d}_{\mathrm{dev}}))\cap 
\mathrm{BV}([0,T];\mathrm{Meas}(\barOmega;\R^{d\ti d}_\mathrm{dev})),
\\&\zeta\in\mathrm{B}([0,T];W^{1,r}(\Omega))\cap H^1(0,T;L^2(\Omega))
\cap C([0,T]{\times}\barOmega)
%\cap L^2(0,T;H^2(\Omega))
\intertext{such that also}
&e_\mathrm{el}=e(u{+}\uD)-\pi\in\mathrm{B}([0,T];H^1(\Omega;\R^{d\times d}))
\ \ \text{ and }
%\ \ \zeta\in L^2(0,T;H^2(\Omega)),
\\&\label{H2-zeta}
\mathrm{div}\big((1{+}\eps|\nabla\zeta|^{r-2})\nabla\zeta\big)\in L^2(Q)
\end{align}
\end{subequations}
is called a weak solution to the initial-boundary-value problem 
\eqref{plast-dam}--\eqref{plast-dam-IC} with the 
substitution \eqref{subst-zero-Dirichlet} if:\\ 
\Item{(i)}{the semi-stability}
\begin{subequations}\label{def-of-weak-sln}
\begin{align}\nonumber\\[-2.5em]
&\label{semi-stab}
\scrE(t,u(t),\pi(t),\zeta(t))\le\scrE(t,\wt u,\wt\pi,\zeta(t))
+\scrR(\zeta(t);\wt\pi-\pi(t),0)
\end{align}
\Item{}{ holds for all $t\in[0,T]$ 
and for all $(\wt u,\wt\pi)\in\mathrm{BD}(\barOmega;\R^d){\ti}
\mathrm{Meas}(\barOmega;\R^{d\ti d}_{\mathrm{dev}})$ with $u
%|_{\GD}^{}\!
\ootimes\vec{n}\d S\!+\pi=0$ on $\GD$ and with  
$e(u){-}\pi\!\in\!H^1(\Omega;\R^{d\ti d}_{\mathrm{sym}})$,}
\Item{(ii)}{the variational inequality}
\begin{align}\label{flow-rule-weak}
&\int_Qa(v)+\Big(\frac12\mathbb C'(\zeta)
%(e(u){-}\pi):(e(u){-}\pi)
e_\mathrm{el}:e_\mathrm{el}
-\kappa\,
\mathrm{div}\big((1{+}\eps|\nabla\zeta|^{r-2})\nabla\zeta\big)
\\[-.7em]&\nonumber
%\Delta\zeta
\hspace{11em}-b'(\zeta)+\xi\Big)(v-\DT\zeta)
%+\kappa\big((1+\eps|\nabla\zeta|^{r-2})\nabla\zeta\big)\cdot\nabla(v-\DT\zeta)
\dd x\dd t
\ge\int_Qa(\DT\zeta)\dd x\dd t,
%\\[.2em]
\end{align}
\Item{}{holds for all $v\!\in\!L^2(Q)$ and 
%for 
some $\xi\!\in\!L^2(Q)$ such that 
$\xi\!\in\!N_{[0,1]}(\zeta)$ a.e.\ on $Q$,}
\Item{(iii)}{the energy equality 
%\COMMENT{HONZO, VSIMNI SI ZE $\partial_t\scrE$ V \eqref{engr-bal} JE TED VLASTNE DOCELA SLOZITY VYRAZ -- POCITAS HO DOBRE?}
}
\begin{align}\label{engr-bal}
&\scrE(T,u(T),\pi(T),\zeta(T))
+\int_{[0,T]\times\barOmega}\big[\Indic_{S(\zeta)}^*(\DT\pi)\big](\d x\d t)
+\int_Q\widehat a(\DT\zeta)\dd x\d t
\\[-.4em]&\nonumber
\hspace{10em}=
\scrE(0,u_0,\pi_0,\zeta_0)
+\int_0^T\!\!\partial_t\scrE(t,u(t),\pi(t),\zeta(t))\dd t.
\end{align}
\Item{}{holds with $\widehat a:\R\to\R$ being the single-valued, continuous 
function defined by $\widehat a(z):=z\partial a(z)$.}
\Item{(iv)}{and also the initial conditions \eqref{plast-dam-IC} hold.} 
\end{subequations}
\end{definition}

Let us note that, counting cancellation of some terms in 
$\scrE(t,u(t),\pi(t),\zeta(t))-\scrE(t,\wt u,\wt\pi,\zeta(t))$, 
the semi-stability \eq{semi-stab} means that 
\begin{align}\label{semi-stab+}
&\int_{\Omega}
\,\frac12\mathbb C(\zeta(t))\big(e(u(t){+}2\uD^{}(t)){-}\pi(t)(t)\big):
\big(e(u(t)){-}\pi(t)\big)
\\[-.5em]\nonumber&\quad\ +
%\frac\eps p|e(u(t){+}\uD^{}(t)){-}\pi(t)|^p
\frac12\mathbb H\nabla\big(e(u(t){+}2\uD^{}(t)){-}\pi(t)(t)\big)\Vdots
\nabla\big(e(u(t)){-}\pi(t)\big)
\dd x
\\\nonumber&\ \le\int_{\Omega}\,
\frac12\mathbb C(\zeta(t))\big(e(\wt u{+}2\uD^{}(t)){-}\wt\pi\big):
\big(e(\wt u){-}\wt\pi\big)
\\[-.3em]\nonumber
&\quad\ +
%\frac\eps p|e(\wt u{+}\uD^{}(t)){-}\wt\pi|^p
\frac12\mathbb H\nabla\big(e(\wt u{+}2\uD^{}(t)){-}\wt\pi\big)\Vdots
\nabla\big(e(\wt u){-}\wt\pi\big)\dd x
+\int_{\barOmega}\big[\Indic_{S(\zeta(t))}^*(\wt\pi{-}\pi(t))\big](\d x).\!\!\!
%\\
%&\int_Q\big(
%%\mathbb D e(\DT{u})
%%%+\mathbb C (\zeta)(e(u){-}\pi)+\eps(1{+}|e(u){-}\pi|^2)^{p/2-1})(e(u){-}\pi)
%%+
%\mathbb C (\zeta)e_\mathrm{el}+\eps\big(1{+}|e_\mathrm{el}|^2\big)^{p/2-1}e_\mathrm{el}\big):e(v)-g{\cdot}v\dd x\dd t=0,
%\\[-.2em]
\end{align}
The last integral \eqref{semi-stab+} is not a Lebesgue integral but an 
integral according the measure $\Indic_{B_1}^*(\wt\pi{-}\pi(t))$. Due to the special
ansatz \eqref{ass-S} below, this integral will the total variation $|\wt\pi{-}\pi(t)|$,
namely $\int_{\barOmega}\sY(\zeta(t))|\wt\pi{-}\pi(t)|(\d x)$.
%Later $B_1$ is considered as a unit ball, so that $\Indic_{B_1}^*$ is just a 
%total variation of the measure in question.
Similarly, the integral on the left-hand side of \eqref{engr-bal} 
equals $\int_{[0,T]\times\barOmega}\sY(\zeta)|\DT\pi|(\d x\d t)$.
Further note that, although traces of functions from 
$\mathrm{BD}(\barOmega;\R^d)$ 
are in $L^1(\Gamma;\R^d)$, one has to be aware of
jumps that can occur at the boundary, i.e.\ the measure $e(u)$ may
concentrate on the boundary $\Gamma$. Thus, the classical boundary
condition $u=0$ on $\GD$ arising 
%standardly 
by the additive shift 
\eqref{subst-zero-Dirichlet2}
%of $u$ 
is replaced by the more involved relation $u\ootimes \vec{n}\d S\!+\pi=0$ on 
$\GD$ in \eq{seismic-U+}. This relation has to be understood as an
equality of measures on $\GD$: 
\[
\forall\, \text{measurable } A\subset \GD: \quad 
\int_{A} u\ootimes \vec{n}\d S = \int_A \d \pi = \pi(A).
\]
The relation simply means that any jump of $u$ on the boundary has to
be due to a localized plastic deformation. Cf.\ \cite{DaDeMo06QEPL}
for analytical details. Eventually, let us comment the last term in \eqref{engr-bal}
which, in view of \eqref{seismic-E+}, involves the expression
\begin{align}\label{DTE}
&\partial_t\scrE(t,u,\pi,\zeta)=
\int_{\Omega}\,\mathbb C(\zeta)\big(e(u{+}\uD(t)){-}\pi\big):e(\DTuD(t))
\\[-.6em]\nonumber
&\qquad\qquad\ +\mathbb H\nabla(e(u{+}\uD(t)){-}\pi)\Vdots\nabla e(\DTuD(t))
-\DT g(t){\cdot}u\dd x
-\!\int_{\GN}\!\!\DT\frM(t){\cdot}u\dd S.
\end{align}

Let us collect the assumptions on the data and on the loading we will rely on,
some of them being already mentioned above:
\begin{subequations}\label{ass}
\begin{align}\label{ass-Omega}
&\Omega\subset\R^d\ \text{bounded $C^2$-domain, 
%smooth or convex domain,
$\GD$ has a ($d{-}2)$ dimensional $C^2$-boundary,
}\!\!\!\!\!
\\\label{ass-a}&
a:\R\to\R\text{ convex, smooth on $\R{\setminus}\{0\}$, $\ a(0)=0$, and }
\\&\nonumber\qquad \exists\,\epsilon>0\ \forall z\!\in\!\R:
\ \ \epsilon|z|^2\le a(z)\le(1{+}|z|^2)/\epsilon, 
\\&
b:[0,1]\to\R\text{ continuously differentiable, non-decreasing, concave},
%\ \ \ b'(1)=0,
\label{ass-b}\\&\label{ass-C}
\mathbb C:[0,1]\to\R^{d\times d\times d\times d}\text{ continuously differentiable,
positive-semidefinite,
%-valued,
}\!\!\!
\\\nonumber&\qquad
\forall\,i,j,k,l=1,\dots,d:\ \ 
\mathbb C_{ijkl}(\cdot)=\mathbb C_{jikl}(\cdot)=\mathbb C_{klij}(\cdot),
\\&\nonumber\qquad\forall\,e\!\in\!\R^{d\times d}_\mathrm{sym}:
\ \mathbb C(\cdot)e{:}e:[0,1]\to\R\text{ non-decreasing, convex},
%\ \ \mathbb C'(0)=0,
\\&\nonumber\qquad\exists\,\mathbb{C}^{}_{\mbox{\tiny\rm D}}(\zeta),\,c^{}_{\mbox{\tiny\rm S}}(\zeta):\quad
\mathbb{C}(\zeta)e:e=\mathbb{C}^{}_{\mbox{\tiny\rm D}}(\zeta)\mathrm{dev}\,e:\mathrm{dev}\,e
+c^{}_{\mbox{\tiny\rm S}}(\zeta)(\mathrm{tr}\,e)^2,
\\&\label{ass-H}\mathbb H\text{ positive definite,}\ \ 
\mathbb H_{ijkl}=\mathbb H_{jikl}=\mathbb H_{klij},
\\[-.5em]\nonumber&\qquad
\exists\,\mathbb{H}^{}_{\mbox{\tiny\rm D}},\,H^{}_{\mbox{\tiny\rm S}}:\qquad
\mathbb{H}\nabla e\Vdots\nabla e=
\mathbb{H}^{}_{\mbox{\tiny\rm D}}\nabla\mathrm{dev}\,e\Vdots\nabla\mathrm{dev}\,e
+H^{}_{\mbox{\tiny\rm S}}\nabla\mathrm{tr}\,e\cdot\nabla\mathrm{tr}\,e,
\\\label{ass-S}&S(\zeta)=\sY(\zeta)B_1,\ \ \ \sY:[0,1]\to(0,\infty)
\text{ continuous nondecreasing},
\\[-.1em]\nonumber&\qquad
\ \ \text{ with }\ B_1\subset\R^{d\times d}_\mathrm{dev}\text{ a unit ball},
\\\label{ass-load}
&\wD\!\in\!W^{1,1}(0,T;H^{3/2}(\GD;\R^d))\text{ and }
\exists\,\uD\!\in\!W^{1,1}(0,T;H^{2}(\Omega;\R^d))
\\&\nonumber\qquad
\text{ satisfying 
\eqref{uD=0-on-GN} and }\uD|^{}_{\GD}=\wD, 
\\&\nonumber\qquad
g\in W^{1,1}(0,T;L^1(\Omega;\R^d)),\ \ \frM\in W^{1,1}(0,T;L^1(\GN;\R^d)),
\\&\nonumber\qquad
\exists\,\sigma_{\mbox{\tiny\rm SL}}:[0,T]\to L^2(\Omega;\R^{d\times d}_{\rm sym})\ \exists\,\alpha>0:
\quad\sigma_{\mbox{\tiny\rm SL}}\vec{n}=g\ \text{ on }[0,T]{\times}\GN\ \text{ and}\!\!\!\!\!
\\\nonumber
&\hspace{4em}{\rm div}\,\sigma_{\mbox{\tiny\rm SL}}+f=0\ \text{ and }\ 
|{\rm dev}\,\sigma_{\mbox{\tiny\rm SL}}|\le\sY(0)-\alpha\ \text{ on }[0,T]{\times}\Omega,
%\\&\qquad\text{with }\ 
%p_0=2d/(d{+}2),\text{ and }p_1=2{-}2/d\text{ if }d>2,
%\text{ or $p_0>1$ and $p_1>1$ for $d=2$},
\\\label{ass-IC-stable}&(u_0,\pi_0,\zeta_0)\in \mathrm{BD}(\barOmega;\R^d){\ti}
\mathrm{Meas}(\barOmega;\R^{d\ti d}_{\mathrm{dev}}){\ti}W^{1,r}(\Omega),\ \ 
\\\nonumber&\qquad0\le\zeta_0\le1\ \text{ a.e.\ on }\Omega,\ \ \text{ and }
\\\nonumber&\qquad
\forall(\wt u,\wt\pi)\in\mathrm{BD}(\barOmega;\R^d){\ti}
\mathrm{Meas}(\barOmega;\R^{d\ti d}_{\mathrm{dev}}),\
\\&\nonumber\hspace{4em}e(\wt u){-}\wt\pi\!\in\!H^1(\Omega;\R^{d\ti d}_{\mathrm{sym}}),\ \
\wt u
%|_{\GD}^{}\!
\ootimes\vec{n}\,\d S\!+\wt\pi=0\text{ on }\GD\!:\ \ 
\\\nonumber&\qquad\qquad
\scrE(0,u_0,\pi_0,\zeta_0)\le\scrE(0,\wt u,\wt\pi,\zeta_0)
+\scrR(\zeta_0;0,\wt\pi{-}\pi_0),
\nonumber\\
&\kappa>0,\ \ \eps>0,
%\ \ p\ge4,
\ \ r>d.
\end{align}
\end{subequations}
The smoothness assumption \eq{ass-Omega} and the ``elastic'' invariance of 
the orthogonal subspaces of deviatoric and volumetric components 
(\ref{ass}d,e) copy the assumptions used in 
\cite{DaDeMo06QEPL} for perfect plasticity in simple materials without damage
in a variant with spatially varying yield stress as in 
\cite{DaDeSo11QECC,FraGia12SSHE,Solo09QEPN}.
%with $p_0=2d/(d{+}2)$ and $p_1=2{-}2/d$ if $d>2$, or 
%$p_0>1$ and $p_1>1$ for $d=2$, or $p_0=1=p_1$  for $d=1$
The stress $\sigma_{\mbox{\tiny\rm SL}}$ in the condition \eqref{ass-load} 
qualifies the loading be $f$ and $g$ in such a way so that the infinite sliding of some 
parts of body is excluded; this is a usual requirement called a safe-load 
condition, connected to perfect plasticity, here adopted to the situation that 
the yield stress $\sY$ may vary with damage similarly as 
in \cite[Remark~2.9]{FraGia12SSHE}. It should be also remarked that 
this safe-load condition works similarly for nonsimple materials.
Further note that \eqref{ass-IC-stable} represents in particular the semi-stability 
of the initial condition and makes, with other assumption, the energy
conservation \eqref{engr-bal} possible. Note also that \eqref{ass-a}
ensures that $\widehat a$ used \eqref{engr-bal} is single-valued
although $a$ itself may be set-valued at 0. In \eqref{ass-S},
one can easily consider a bit more general situation when 
$B_1$ would be convex, closed, and $0\in\text{int}\,B_1$.

%
%\COMMENT{IS $\Omega$ convex sufficient for $H^2$-regularity????}
%
%\COMMENT{IS $\nabla\DT\zeta$ WELL DEFINED?? -- AN $H^2$-regularity of $\zeta$ would help.....}
%
The main analytical result justifying rigorously the model 
\eqref{plast-dam}--\eqref{plast-dam-IC} is:

%If $S(\zeta)=\sY(\zeta)B_1$ with some $s:[0,1]\to\R^+$ continuous,

\medskip

\begin{theorem}
Under the assumptions \eqref{ass}, at least one weak solution to 
the initial-boundary-value problem \eqref{plast-dam}--\eqref{plast-dam-IC}
according to Definition~\ref{weak-sln} does exist.
\end{theorem}

\medskip

We will prove this existence result in Section~\ref{sec-disc} 
by a constructive
time discretisation method, cf.\ Lemma~\ref{lem-disc} with 
Proposition~\ref{prop-conv}, which later in Sections~\ref{sec-comp} and 
\eqref{sec-simul} allows for efficient computer implementation of the model.
The uniqueness of the solution however hardly can be expected.

\medskip

\begin{remark}[The dynamical model]\label{rem-dynam}
\upshape
During fast rupture, inertial effects may be not negligible and
even sometimes an important aspect of the model. Then, \eqref{plast-dam1}
augments by the inertial force $\varrho\DDT{u}$ with $\varrho>0$ 
denoting the mass density as
\begin{align}\label{plast-dam1-with-rho}
&\varrho\DDT{u}-
\mathrm{div}
%\big(
%\mathbb D e(\DT{u})
%%+\mathbb C (\zeta)(e(u){-}\pi)+\eps(1{+}|e(u){-}\pi|^2)^{p/2-1})(e(u){-}\pi)
%+
%\mathbb C (\zeta)e_\mathrm{el}
%%+\eps|e_\mathrm{el}|^{p-2}e_\mathrm{el}
%-\mathrm{div}\,\mathfrak{h}\big)
\,\sigma=g
\ \ \ \text{ with }\ \
\sigma=\mathbb C (\zeta)e_\mathrm{el}-\mathrm{div}\,\mathfrak{h}.
\end{align}
Relying on that the inertial term $\varrho\DDT{u}$ is controlled in the space
$L^2(0,T;H^2(\Omega;\R^3)^*)\,$ $\cap$ $\,C_{\rm weak}([0,T];L^2(\Omega;\R^3))$
or actually even in a slightly better space counting that ${\rm dev}\,\sigma\in L^\infty(Q;\R^{d\times d}_{\rm sym})$, 
the weak formulation 
of \eqref{plast-dam1-with-rho} arising by double by-part integration
in time should accompany \eqref{def-of-weak-sln} with $\scrE$ 
augmented by the inertial energy $\int_\Omega\frac\varrho2|\DT u|^2\d x$
but with \eqref{semi-stab} holding only a.e.\ on $[0,T]$ and 
\eqref{engr-bal} only as an inequality.
%if a Kelvin-Voigt-type rheology with the viscosity acting on 
%the elastic strain rather than the elastic strain 
%allows still for development on the sharp shear bands 
%and simultaneously for rigorous analysis
%in the isothermal case, although 
%
%Moreover, the semistability \eqref{semi-stab} is to be
%written rather as $\int_\Omega \mathbb C e_{\rm el}{:}e_{\rm el}
%+...$
The functional in \eqref{minimization-1} then augments by $\varrho\tau^{-2}|u-2u_\tau^{k-1}+u_\tau^{k-2}|^2/2$.
%\COMMENT{TO CHECK OTHER CHANGES??}
Actually, it seems a matter of a physically-explainable fact 
that some difficulties the energy conservation occurs probably 
due to integration of elastic waves with nonlinearly responding shear bands
even if a Kelvin-Voigt-type visco-elastic rheology would be involved, 
%occurs in the anisothermal expansion of the model because 
%the energy conservation seems not , 
cf.\ also \cite[Remark 6]{Roub13TPP}. 
%
%\COMMENT{CAN ONE PROVE CONVERGENCE HERE WITHOUT IT??}
In this dynamical case, the fast damage phases and subsequent fast 
plastic slips, called (tectonic) \emph{earthquakes}, typically emit elastic 
(seismic) waves. However, although 
%our model can easily combine waves with inelastic processes 
some justification on theoretical level, the computational modelling 
requires fine special techniques to suppress e.g.\ parasitic numerical 
attenuation and the direct combination of elastic waves with the inelastic 
processes 
%as in \eq{eq6:plast-dam} 
is difficult.
\end{remark}

\medskip

\begin{remark}[A non-Hookean model]\label{rem-nonlin}
\upshape
The concept of nonsimple materials allows an important generalization 
that $\scrE(t,\cdot,\zeta,\cdot)$ is not quadratic and even nonconvex.
More specifically, instead of the coercive term 
$(e_\mathrm{el},\zeta)\mapsto\mathbb{C}(\zeta)e_\mathrm{el}{:}e_\mathrm{el}
=\frac12\lambda(\zeta)I_1^2+\mu(\zeta)I_2$ as used also here in 
%\COMMENT{HERE LINK TO Sect.~\ref{sec-comp}}
\eqref{C-ass-implem} below, \cite{LyaMya84BECS} 
%\COMMENT{ALSO TO CHECK for some theoretical justification:
%BUDIANSKY, B. and O’CONNELL, R.J. (1976), Elastic moduli of a
%cracked solid. Int. J. Solids Struct., 12, 81-97.
%BEN-ZION, Y. and V. LYAKHOVSKY (2006), Analysis of Aftershocks
%in a Lithospheric Model with Seismogenic Zone Governed by
%Damage Rheology, Geophys. J. Int., 165, 197–210, doi:
%10.1111/j.1365-246X.2006.02878.x.}
proposed 
\begin{align}\label{non-Hookean-model}
(e_\mathrm{el},\zeta)\mapsto
\frac12\lambda(\zeta)I_1^2+\mu(\zeta)I_2-\gamma(\zeta)I_1\sqrt{I_2}
\ \ \text{ with }\ 
I_1={\rm tr}\,e_\mathrm{el},\ 
%\text{ and }
\ I_2=|e_\mathrm{el}|^2.
\end{align}
The elastic stress is then 
$(\lambda(\zeta){-}\gamma(\zeta)\sqrt{I_2}){\rm tr}\,e_\mathrm{el}+
(2\mu(\zeta){-}\gamma(\zeta)I_1/\sqrt{I_2})e_\mathrm{el}$,
while the driving stress for damage is $\sigma_{\rm dam}
=\frac12\lambda'(\zeta)I_1^2+\mu'(\zeta)I_2-\gamma'(\zeta)I_1/\sqrt{I_2}$
and can now be positive even without the contribution of the $b$-term.
%which then represent another (even dominant) mechanism for healing.
Such a model is widely used in geophysics where it is believed to 
be responsible for instability of heavily damaged rocks and leads to healing 
even without the $b$-term used in our model, but where it
is used without the $\mathbb H$-term and thus without any rigorous
justification of such models, cf.\ e.g.\ 
%\cite{HaLyBZ11ESED,LyBZAg97DDFF,LRWS97NEBD} 
\cite{HaLyBZ11ESED,LRWS97NEBD} and references there. 
To preserve coercivity of the model due to boundary conditions and 
the $\mathbb H$-term, one can think about a certain softening 
under very large strain by replacing 2-homogeneous form \eqref{non-Hookean-model}
by an energy with only a linear growth
\begin{align}
(e_\mathrm{el},\zeta)\mapsto\frac{\lambda(\zeta)I_1^2
+2\mu(\zeta)I_2-2\gamma(\zeta)I_1\sqrt{I_2}}{\sqrt{4+\epsilon I_2}}
%\quad\text{ with $\ I_1={\rm tr}\,e_\mathrm{el},\ 
%%$ and $
%\ I_2=|e_\mathrm{el}|^2$,}
\end{align}
with $\epsilon>0$ presumably small. A certain conceptual inconsistency remains 
in damage-dependence of $\mathbb C$ but not of $\mathbb H$, although 
$\mathbb H$ is assumed to be only small in applications. Note that 
\eqref{minimization-1} then represents a coercive but non-convex minimization
problem and one should seek a global minimizer to ensure \eqref{semi-stab-disc}.
The nonsimple-material concept allows for a simple 
modification of the convergence proof in semistability and in the damage flow
by compactness: more specifically, the binomial trick in \eqref{semi-stab-proof}
is applied only to the dissipation and the $\mathbb H$-terms, while 
\eqref{uniqueness-of-stress} is even simpler because 
$\mathbb C'(\zeta)e_{\rm el}$ is now bounded in $L^\infty(\Omega;\R^{d\times d})$.
% cf.\ Steps 3 and 4 ............................ \COMMENT{modification
%of the strong convergence \eqref{e-strong} of the elastic strain needed ????}
%in the proof of Proposition~\ref{prop-conv} below.
\end{remark}

\medskip

\begin{remark}[A simple-material model]\label{rem-simple}
\upshape
Considering $\mathbb H=0$ would bring various difficulties. In particular,
the $L^2(Q)$-estimate of the driving force 
$\frac12\mathbb C'(\zeta)e_\mathrm{el}{:}e_\mathrm{el}$, which would
need a regularity of $e_\mathrm{el}$ that however does not seem available
for plasticity models without hardening, would become problematic. Note that 
the higher integrability of $e_\mathrm{el}\otimes e_\mathrm{el}$ will be used
e.g.\ in \eqref{uniqueness-of-stress} and in \eqref{engr-bal-dam} too.
One should note that the alternative idea to consider
a nonlinear damage independent contribution to the stress of the type
$+\eps|e_\mathrm{el}|^2e_\mathrm{el}$ would not allow to use the 
binomial trick in the Step 3 in the proof of Proposition~\ref{prop-conv}
below, while the strong convergence of $e_\mathrm{el}$ seems also not
obvious to prove.
%with the need of strong convergence of in $e_\mathrm{el}$
A certain possibility might be in considering a 
visco-elastic Kelvin-Voigt model with the stress 
$\mathbb D(\zeta)\DT e_\mathrm{el}+
\mathbb C(\zeta,e_\mathrm{el})$ with a nonlinear, monotone 
$\mathbb C(\zeta,\cdot)$ having at most the growth
$|\mathbb C(\zeta,e_\mathrm{el})|\le C(1+|e_\mathrm{el}|^{1/2})$
so that $\int_0^1\partial_\zeta\mathbb C(\zeta,t e_\mathrm{el})\dd t$
can still be estimated in $L^2(Q)$ due to the $\mathbb D$-term
which can even depend on $\zeta$ as in \cite{MiRoZe10CDEV}.
\end{remark}

\section{The discretisation, its stability and convergence}\label{sec-disc}
%        ~~~~~~~~~~~~~~~~~~~~~~~~~~~~~~~~~~~~~~~~~~~~~~~~~

To implement the initial-boundary-value problem 
\eqref{plast-dam}--\eqref{plast-dam-IC} computationally,
we need to make a time and space discretisation.

Let us first make only a time discretisation with, for notational simplicity,
a constant time step $\tau>0$. As the inertial effects are not considered 
and thus the system is only 1st-order in time,
the dependence of $\tau>0$ on the time levels is easy to consider for
numerical analysis and to implement (as actually used in 
Section~\ref{sec-comp} below).

As $\scrE$ is convex in terms of $(u,\pi)$ and separately in $\zeta$ too, and 
also as $\scrR$ additively splits $(\DT u,\DT\pi)$ from $\DT\zeta$,
the natural \emph{fractional-step strategy} leading to an efficient and numerically 
stable semi-implicit formula follows this splitting 
$(u,\pi)$ from $\zeta$. More specifically, 
%forgetting for simplicity the possible shift in case of non-homogeneous 
%Dirichlet conditions, 
it reads as
\begin{subequations}\label{plast-dam-disc}
\begin{align}\label{eq6:plast-dam1-disc}
&
%\varrho\frac{u_\tau^k{-}2u_\tau^{k-1}{+}u_\tau^{k-2}}{\tau^2}-
\mathrm{div}\Big(
%\mathbb D e\big(\frac{u_\tau^k{-}u_\tau^{k-1}}{\tau}\big)
%+\mathbb C (\zeta)(e(u){-}\pi)+\eps(1{+}|e(u){-}\pi|^2)^{p/2-1})(e(u){-}\pi)
%+
\mathbb C(\zeta_\tau^{k-1})e_{\mathrm{el},\tau}^k
%\\[-.5em]&
%\hspace{18em}
%+\eps
%%\big(1{+}|e_{\mathrm{el},\tau}^k|^2\big)^{p/2-1}
%|e_{\mathrm{el},\tau}^k|^{p-2}e_{\mathrm{el},\tau}^k
-\mathrm{div}\,\mathfrak{h}_\tau^k\Big)+g_\tau^k=0
\\[-.1em]\nonumber
&\hspace{2em}\text{with }\ e_{\mathrm{el},\tau}^k=
e(u_\tau^k{+}\uD(k\tau)){-}\pi_\tau^k,\ \ \ \
\mathfrak{h}_\tau^k=\mathbb H\nabla e_{\mathrm{el},\tau}^k,\ \ \ \ 
g_\tau^k:=g(k\tau),
\\[-.1em]\label{plast-dam12-disc}
&
%\partial\Indic_{S(\zeta_\tau^{k-1})}^*
N_{S(\zeta_\tau^{k-1})}^{}\Big(\frac{\pi_\tau^k{-}\pi_\tau^{k-1}}{\tau}\Big)
%+\mathbb H\pi_\tau^k
%+\mathrm{dev}\,\mathbb C (\zeta)\pi\ni\mathrm{dev}\,\mathbb C (\zeta) e(u)
\ni\mathrm{dev}\Big(\mathbb C (\zeta_\tau^{k-1})e_{\mathrm{el},\tau}^k
%+\eps|e_{\mathrm{el},\tau}^k|^{p-2}e_{\mathrm{el},\tau}^k
-\mathrm{div}\,\mathfrak{h}_\tau^k\Big),
%\ \ \ \ \text{ with }
%\ \ \text{ and }\mathfrak{h}_\tau^k=\mathbb H\nabla e_{\mathrm{el},\tau}^k
\\[-.1em]\label{plast-dam13-disc}
&\partial a\Big(\frac{\zeta_\tau^k{-}\zeta_\tau^{k-1}}{\tau}\Big)
+\frac12\mathbb C'(\zeta_\tau^k)
%(e(u){-}\pi):(e(u){-}\pi)
e_{\mathrm{el},\tau}^k:e_{\mathrm{el},\tau}^k
\\[-.1em]&\nonumber\qquad\qquad\qquad\
-\kappa\,
\mathrm{div}\big((1{+}\eps|\nabla\zeta_\tau^k|^{r-2})\nabla\zeta_\tau^k\big)
%\Delta\zeta_\tau^k
+N_{[0,1]}^{}(\zeta_\tau^k)\ni b'(\zeta_\tau^k),
\end{align}\end{subequations}
together with the corresponding boundary conditions 
\begin{subequations}\label{BC-disc}
\begin{align}
&u_\tau^k=
%\wDtau^k
0 &&\text{on }\GD
%\ \ \text{with }\ \wDtau^k:=\wD(k\tau)
,
\\&
%\!\!\!\left.\begin{array}{l}
\!\big(\mathbb C (\zeta_\tau^{k-1})e_{\mathrm{el},\tau}^k
%+\eps|e_{\mathrm{el},\tau}^k|^{p-2}e_{\mathrm{el},\tau}^k
-\mathrm{div}\,\mathfrak{h}_\tau^k\big)\cdot\vec{n}
-\divS(\mathfrak{h}_\tau^k\vec{n})=\frM_\tau^k
%,\\[.2em]%\ \text{ and }\ 
%\mathfrak{h}_\tau^k{:}(\vec{n}\otimes\vec{n})=0\end{array}\right\}
&&\text{on }\GN\ \ \text{with }\ \frM_\tau^k:=\frM(k\tau),
\\\label{BC-disc-zeta}
&
%(1{+}\eps|\nabla\zeta_\tau^k|^{r-2})
\nabla\zeta_\tau^k\cdot\vec{n}=0\ \ \ \ \text{ and }\ \ \ \ 
\mathfrak{h}_\tau^k{:}(\vec{n}\otimes\vec{n})=0&&\text{on }\Gamma,
\end{align}\end{subequations}
to be solved first for $(u_\tau^k,\pi_\tau^k)$ from 
(\ref{plast-dam-disc}a,b)--(\ref{BC-disc}a,b) and then for $\zeta_\tau^k$ from 
\eq{plast-dam13-disc}--\eq{BC-disc-zeta} recursively for $k=1,...,T/\tau$. 
Both these boundary-value problems
have potentials and thus leads to minimization problems.
Moreover, as $\mathbb C'$ and $-b'$ are nondecreasing
(again with respect to the L\"owner's ordering) and $a$ is convex 
as assumed in (\ref{ass}),
%\footnote{The monotonicity of  $\mathbb C '$ is with respect to the L\"owner's 
%ordering, cf.\ Example~\ref{exa-Lowner} on p.\,\pageref{exa-Lowner}.}, then 
both these boundary-value problems leads to convex variational problems, 
cf.\ \eqref{minimization} below.
%; more specifically, as $a$ is uniformly convex 
%\COMMENT{WILL IT TRUE IN OUR SIMULATIONS?} 
%Note that the formulas 
%\eq{eq6:engr-ineq-u-z}--\eq{eq6:engr-ineq-dissip} now modify by replacing 
%$u_\tau^k$ and $\baru_\tau$ with $\zeta_\tau^{k-1}$ and $\ul\zeta_\tau$,
%respectively.

Let us define the piecewise affine interpolant $u_\tau$  by 
\begin{subequations}\label{def-of-interpolants}
\begin{align}
&u_\tau(t):=
\frac{t-(k{-}1)\tau}\tau u_\tau^k
+\frac{k\tau-t}\tau u_\tau^{k-1}
\quad\text{ for $t\in[(k{-}1)\tau,k\tau]\ $}
\intertext{with $\ k=0,...,T/\tau$. Besides, we define also the 
left-continuous piecewise constant interpolant $\baru_\tau$ and  the 
right-continuous piecewise constant interpolant $\underline u_\tau$
by}
&\baru_\tau(t):=u_\tau^k
\qquad\qquad\text{ for $t\in((k{-}1)\tau,k\tau]\ $,\ \ 
$k=1,...,T/\tau$},
\\
&\underline u_\tau(t):=u_\tau^{k-1}
\qquad\quad\ \text{for $t\in[(k{-}1)\tau,k\tau)\ $,\ \ 
$k=1,...,T/\tau$}.
\end{align}
\end{subequations}
Similarly, we define also $\pi_\tau$, $\barpi_\tau$, 
$\underline\pi_\tau$, $\barzeta_\tau$, $\zeta_\tau$, $\bar g_\tau$, 
%$\bar\frM_\tau$, 
etc. 

\medskip

\begin{lemma}[Existence and stability of discrete solutions]\label{lem-disc}
The recursive boun\-dary-value problem \eq{plast-dam-disc}--\eq{BC-disc}
has a weak solution 
%$(u_\tau^k,\pi_\tau^k,\zeta_\tau^k)\in
%\mathrm{BD}(\barOmega;\R^d){\ti}W^{1,r}(\Omega){\ti}
%\mathrm{Meas}(\barOmega;\R^{d\ti d}_{\mathrm{dev}})$
$(u_\tau^k,\pi_\tau^k,\zeta_\tau^k)$ with $u_\tau^k\in
\mathrm{BD}(\barOmega;\R^d)$, $\pi_\tau^k\in W^{1,r}(\Omega)$,
and $\zeta_\tau^k\in\mathrm{Meas}(\barOmega;\R^{d\ti d}_{\mathrm{dev}})$
with $e_{\mathrm{el},\tau}^k=e(u_\tau^k){-}\pi_\tau^k\in H^1(\Omega;\R^{d\ti d}_{\mathrm{sym}})$ 
satisfying the a-priori estimates 
\begin{subequations}\label{est}
\begin{align}
&\big\|\baru_\tau\big\|_{L^\infty(0,T;\mathrm{BD}(\barOmega;\R^d))}\le C,
\\&\big\|\barpi_\tau\big\|_{L^\infty(0,T;\mathrm{Meas}(\barOmega;\R^{d\ti d}_{\mathrm{dev}}))\,\cap\,
\mathrm{BV}([0,T];L^1(\Omega;\R^{d\ti d}_{\mathrm{dev}}))}\le C,
\\&\big\|e(u_\tau){-}\pi_\tau\big\|_{L^\infty(0,T;H^1(\Omega;\R^{d\ti d}_{\mathrm{sym}}))}\le C,
\\\label{est-Delta-zeta}&\big\|\zeta_\tau\big\|_{L^\infty(0,T;W^{1,r}(\Omega))\,\cap\,H^1(0,T;L^2(\Omega))}\le C,
\\\label{est-Delta-zeta+}
&\big\|\mathrm{div}((1{+}\eps|\nabla\barzeta_\tau|^{r-2})\nabla\barzeta_\tau)\big\|_{L^2(Q)}\le C.
\end{align}
\end{subequations}
\end{lemma}

\begin{proof}
The existence of weak solutions to \eqref{plast-dam-disc} can be justified by
the direct method when realizing the variational structure:
the boundary-value problem (\ref{plast-dam-disc}a,b)--(\ref{BC-disc}a,b) represents 
a minimization problem 
\begin{subequations}\label{minimization}
\begin{align}\label{minimization-1}
\left\{\begin{array}{ll}\text{Minimize}&
(u,\pi)\mapsto\scrE(k\tau,u,\pi,\zeta_\tau^{k-1})+\scrR(\zeta_\tau^{k-1};\pi{-}\pi_\tau^{k-1},0)
\\[.3em]\text{subject to}&u\in\mathrm{BD}(\barOmega;\R^d),\ \ \pi\in
\mathrm{Meas}(\barOmega;\R^{d\ti d}_{\mathrm{dev}}),\
\\[.2em]&\,e(u){-}\pi\!\in\!H^1(\Omega;\R^{d\ti d}_{\mathrm{sym}}),\ \ 
u\ootimes\vec{n}\d S\!+\pi=0\text{ on }\GD,\end{array}\right.
\intertext{while the boundary-value problem \eqref{plast-dam13-disc}--\eqref{BC-disc-zeta}
represents a minimization problem}
\label{minimization-2}
\left\{\begin{array}{ll}\text{Minimize}&\displaystyle{
\zeta\mapsto\scrE(k\tau,u_\tau^k,\pi_\tau^k,\zeta)
+\tau\scrR\Big(\zeta_\tau^{k-1};0,\frac{\zeta{-}\zeta_\tau^{k-1}}\tau\Big)}
\\\text{subject to}&\zeta\in W^{1,r}(\Omega),\ \ 0\le\zeta\le1\ \text{ on }\Omega,
\end{array}\hspace{1.2em}\right.
\end{align}
\end{subequations}
whose solutions do exist by coercivity, convexity, and lower semicontinuity arguments.
Here the safe-load qualification \eqref{ass-load} of $f$ and $g$ is to be used.

Further, we test \eqref{plast-dam-disc} respectively by $u_\tau^k{-}u_\tau^{k-1}$,
$\pi_\tau^k{-}\pi_\tau^{k-1}$, and $\zeta_\tau^k{-}\zeta_\tau^{k-1}$.
Relying on the convexity of $\scrE(k\tau,\cdot,\cdot,\zeta_\tau^{k-1})$
and of $\scrE(k\tau,u_\tau^k,\pi_\tau^k,\cdot)$, we obtain the estimates
\begin{subequations}
\begin{align}
&\scrE(k\tau,u_\tau^k,\pi_\tau^k,\zeta_\tau^{k-1})
+\int_{\barOmega}
%\big[\Indic_{S(\zeta_\tau^{k-1})}^*(\pi_\tau^k{-}\pi_\tau^{k-1})\big]
\sY(\zeta_\tau^{k-1})|\pi_\tau^k{-}\pi_\tau^{k-1}|(\d x)\le
\scrE(k\tau,u_\tau^{k-1},\pi_\tau^{k-1},\zeta_\tau^{k-1}),\!\!\!
\label{est1}\\&\scrE(k\tau,u_\tau^k,\pi_\tau^k,\zeta_\tau^k)
+\int_\Omega\widehat a(\zeta_\tau^k{-}\zeta_\tau^{k-1})\dd x
\le
\scrE(k\tau,u_\tau^k,\pi_\tau^k,\zeta_\tau^{k-1})
\label{est2}\end{align}
\end{subequations}
with $\widehat a$ from \eqref{engr-bal}.
By summing these estimates, we can enjoy the cancellation of the terms 
$\scrE(k\tau,u_\tau^k,\pi_\tau^k,\zeta_\tau^{k-1})$ in \eqref{est1}
and \eqref{est2}, and we thus obtain
\begin{align}\label{enegr-est}
&\scrE(k\tau,u_\tau^k,\pi_\tau^k,\zeta_\tau^k)
+\widehat\scrR(\zeta_\tau^{k-1};\pi_\tau^k{-}\pi_\tau^{k-1},\zeta_\tau^k{-}\zeta_\tau^{k-1})
\le\scrE(k\tau,u_\tau^{k-1},\pi_\tau^{k-1},\zeta_\tau^{k-1})
\\&\nonumber\qquad\ \
=\scrE((k{-}1)\tau,u_\tau^{k-1},\pi_\tau^{k-1},\zeta_\tau^{k-1})
+\int_{(k-1)\tau}^{k\tau}\!\!\partial_t\scrE(t,u_\tau^{k-1},\pi_\tau^{k-1},\zeta_\tau^{k-1})\dd t
\end{align}
with the dissipation rate $\widehat\scrR$ defined as
\begin{align}\label{dis-rate}
\widehat\scrR(\zeta;\DT\pi,\DT\zeta):=
\int_{\barOmega}
%\big[\Indic_{S(\zeta)}^*(\DT\pi)\big]
\sY(\zeta)|\DT\pi|(\d x)
+\int_\Omega\widehat a(\DT\zeta)\dd x\qquad\text{ with }\ 
\widehat a(\DT\zeta)=\DT\zeta\partial a(\DT\zeta).
\end{align}
By summing \eqref{enegr-est} over $k$ we enjoy a ``telescopic'' cancellation
effect. Realizing \eqref{DTE} and \eqref{ass-load},
%\COMMENT{..... CHECK ASSUMPTIONS ON $\partial_t\scrE$....\eqref{ass-load}....}
by the discrete Gronwall inequality, we obtain (\ref{est}a--d).

Having estimated $\partial a(\DT\zeta_\tau)
+\frac12\mathbb C'(\barzeta)\bare_{\mathrm{el},\tau}:\bare_{\mathrm{el},\tau}
-b'(\barzeta_\tau)$ as a bounded set in $L^2(Q)$ uniformly with respect to 
$\tau>0$, we can estimate also $%\Delta\zeta_\tau^k
\mathrm{div}((1{+}\eps|\nabla\zeta_\tau^k|^{r-2})\nabla\zeta_\tau^k)
%-N_{[0,1]}^{}(\zeta_\tau^k)
$ in the same space. 
%The $H^2$-regularity of the damage follows by a test of \eq{plast-dam13-disc}
%by 
%%$-\Delta\zeta_\tau^k$
For this, we test \eqref{plast-dam13-disc} by
$-\mathrm{div}((1{+}\eps|\nabla\zeta_\tau^k|^{r-2})\nabla\zeta_\tau^k)$. 
Here, the 
important ingredient is, written rather formally, the following estimate
\begin{align*}
&\int_\Omega\!N_{[0,1]}^{}(\zeta_\tau^k)\big(
%-\Delta\zeta_\tau^k
{-}\mathrm{div}((1{+}\eps|\nabla\zeta_\tau^k|^{r-2})\nabla\zeta_\tau^k)\big)
\dd x
\\[-.4em]&\hspace{6em}
=-\int_\Omega \partial\Indic_{[0,1]}^{}(\zeta_\tau^k)\big(
%\Delta\zeta_\tau^k
\mathrm{div}((1{+}\eps|\nabla\zeta_\tau^k|^{r-2})\nabla\zeta_\tau^k)\big)\dd x
\\\nonumber&\hspace{6em}
=\int_\Omega\nabla\big(\partial\Indic_{[0,1]}^{}(\zeta_\tau^k)\big)
{\cdot}(1{+}\eps|\nabla\zeta_\tau^k|^{r-2})\nabla\zeta_\tau^k\dd x
\\&\hspace{6em}
=\int_\Omega \partial^2\Indic_{[0,1]}^{}(\zeta_\tau^k)\cdot\nabla\zeta_\tau^k{\cdot}
(1{+}\eps|\nabla\zeta_\tau^k|^{r-2})\nabla\zeta_\tau^k\dd x\ge0
\nonumber
\end{align*}
which is due to the positive-semidefiniteness of the (generalized) Jacobian
$\partial^2\Indic_{[0,1]}^{}$ of the convex function $\Indic_{[0,1]}^{}$
and which is to be proved rigorously by a mollification of $\Indic_{[0,1]}^{}$, 
cf.\ \cite[Lemma~1]{RouSte??MSMA} for technical details. Thus we obtain 
\eqref{est-Delta-zeta+}.
% and, due to the qualification \eqref{ass-Omega} of $\Omega$ and the standard 
%$H^2$-regularity theory for Laplace operator with homogeneous Neumann boundary 
%conditions, we obtain also $\barzeta_\tau$ estimated in $L^2(0,T;H^2(\Omega))$.
\end{proof}

\medskip

\begin{lemma}[Discrete analog of \eqref{def-of-weak-sln}]\label{lem-disc-2}
With the notation \eqref{def-of-interpolants} and 
$\bare_{\mathrm{el},\tau}=e(\baru_\tau{+}\baruDtau)-\barpi_\tau$, the discrete 
solution obtained by the recursive scheme 
\eqref{plast-dam-disc}--\eqref{BC-disc} satisfies:
\begin{subequations}\begin{align}
\label{semi-stab-disc}
&\scrE(t,\baru_\tau(t),\barpi_\tau(t),\underline\zeta_\tau(t))\le
\scrE(t,\wt u,\wt\pi,\underline\zeta_\tau(t))
+\scrR(\underline\zeta_\tau(t);\wt\pi-\barpi_\tau(t),0)
\intertext{for all $t\in[0,T]$ and all $(\wt u,\wt\pi)$ as in 
\eqref{semi-stab}, and}
\label{flow-rule-disc}
&\int_Qa(v)+\Big(\frac12\mathbb C'(\underline\zeta_\tau)
%(e(u){-}\pi):(e(u){-}\pi)
\bare_\mathrm{el,\tau}:\bare_\mathrm{el,\tau}
-\kappa\,
\mathrm{div}\big((1{+}\eps|\nabla\barzeta_\tau|^{r-2})\nabla\barzeta_\tau\big)
\\[-.6em]\nonumber
&\hspace{6em}
%\Delta\zeta
-b'(\barzeta_\tau)+\barxi_\tau\Big)(v-\DT\zeta_\tau)
%+\kappa\big((1+\eps|\nabla\zeta|^{r-2})\nabla\zeta\big)\cdot\nabla(v-\DT\zeta)
\dd x\dd t
\ge\int_Qa(\DT\zeta_\tau)\dd x\dd t
%\\[.2em]
\intertext{holds for all $v\in L^2(Q)$ and for some $\barxi_\tau\in L^2(Q)$ 
such that $\barxi_\tau\in N_{[0,1]}(\barzeta_\tau)$ a.e.\ on $Q$, and 
eventually the energy (im)balance holds:}
\label{engr-bal-disc}
&\scrE(T,u_\tau(T),\pi_\tau(T),\zeta_\tau(T))
+\int_0^T\!\widehat\scrR(\underline\zeta_\tau;\DT\pi_\tau,\DT\zeta_\tau)\dd t
%+\int_{[0,T]\times\barOmega}
%\big[\Indic_{S(\underline\zeta_\tau)}^*(\DT\pi_\tau)\big](\d x\d t)
\\[-.5em]&\hspace{6em}
\nonumber
%\hspace{10em}
\le
\scrE(0,u_0,\pi_0,\zeta_0)+\int_0^T\!\!
\partial_t\scrE(t,\underline u_\tau(t),\underline\pi_\tau(t),\underline\zeta_\tau(t))\dd t
\end{align}\end{subequations}
with the overall dissipation rate $\widehat\scrR$ from \eqref{dis-rate}.
Moreover, the a-priori estimate holds:
\begin{align}\label{est-of-xi}
\big\|\barxi_\tau\big\|_{L^2(Q)}\le C.
\end{align}
\end{lemma}

\begin{proof}
The boundary-value problem (\ref{plast-dam-disc}a,b)--(\ref{BC-disc}a,b) represents 
a minimization problem \eqref{minimization-1}
which can be tested by $(u_\tau^{k-1},\pi_\tau^{k-1})$ and, by using a triangle inequality 
facilitated by the 1-homogeneity of $\scrR(\zeta;\cdot,\DT\zeta)$, we obtain 
\eqref{semi-stab-disc}; actually, this is a standard argument in the theory of 
rate-independent processes \cite{Miel05ERIS,MieRou15RIPT,MieThe04RIHM}.

In the case of the boundary-value problem \eqref{plast-dam13-disc}--\eqref{BC-disc-zeta},
the variational inequality \eqref{flow-rule-disc} represents just 
the conventional weak formulation 
of the minimization problem \eqref{minimization-2} summed for all time levels.
Then, \eqref{engr-bal-disc} follows by summing \eqref{enegr-est} for 
$k=1,...,T/\tau$. 

Eventually, the estimate \eqref{est-of-xi} follows by comparison from the inclusion 
$\barxi_\tau\in b'(\barzeta_\tau)
-\frac12\mathbb C'(\underline\zeta_\tau)\bare_\mathrm{el,\tau}:\bare_\mathrm{el,\tau}
+\kappa\,
\mathrm{div}((1{+}\eps|\nabla\barzeta_\tau|^{r-2})\nabla\barzeta_\tau)
%\Delta\zeta
-\partial a(\DT\zeta_\tau)$ and by the already obtained estimates.
%; here again we vitally used $p\ge4$.
\end{proof}

\medskip

\begin{proposition}[Convergence]\label{prop-conv}
Let the assumptions \eqref{ass} be satisfied and the approximate solution 
$(\baru_{\tau},\barpi_{\tau},\barzeta_\tau,\barxi_\tau)$ be
obtained by the recursive scheme \eqref{plast-dam-disc}--\eqref{BC-disc}.
Then there is a subsequence and $(u,\pi,\zeta,\xi)$ such that
\begin{subequations}\label{conv}
\begin{align}\label{conv-u}  
&\baru_{\tau}(t)\to u(t)&&\text{weakly* in }
\mathrm{BD}(\barOmega;\R^d),
%\text{ for any }t\in[0,T],&&&&
\\\label{conv-pi}&\barpi_{\tau}(t)\to\pi(t)&&\text{weakly* in }
\mathrm{Meas}(\barOmega;\R^{d\times d}_\mathrm{dev}),
%\text{ for any }t\in[0,T],
\\\label{conv-eel} 
&\bare_\mathrm{el,\tau}(t)=e(\baru_{\tau}(t)){-}\barpi_{\tau}(t)
\to e(u(t)){-}\pi(t)=e_\mathrm{el}(t)\hspace*{-5.5em}&&\hspace*{5em}
\text{weakly
%strongly 
in }H^1(\Omega;\R^{d\times d}_\mathrm{sym}),\!\!
%\text{ for any }t\!\in\![0,T],\!\!
\\\label{conv-zeta}
&\barzeta_\tau(t)\to\zeta(t)\ \ \text{ and }\ \ 
\underline\zeta_\tau(t)\to\zeta(t)&&\text{weakly in }W^{1,r}(\Omega)
%\text{ for any }t\in[0,T],
%\\
\intertext{holding for  any $t\!\in\![0,T]$, and further also}
\label{conv-zeta-Loo}&
\underline\zeta_\tau\to\zeta&&\text{strongly in }L^\infty(Q),\text{ and}
\\\label{conv-xi}
&\barxi_\tau\to\xi&&\text{weakly in }L^2(Q)
\end{align}\end{subequations}
with $\barxi_\tau$ from Lemma~\ref{lem-disc-2}. Moreover, any $(u,\pi,\zeta)$
obtained by such a way is a weak solution according Definition~\ref{weak-sln}
with $\xi$ in \eqref{flow-rule-weak} taken from \eqref{conv-xi}.
\end{proposition}

\medskip

\begin{proof}
For clarity of exposition, we divide the proof into five particular steps.

\medskip

\noindent \emph{Step 1: Selection of a converging subsequence.} 
By Banach's selection principle, we select a weakly* converging subsequence 
with respect to the norms from the estimates \eqref{est} and \eqref{est-of-xi};
namely, for some $u$, $\pi$, $\zeta$, and $\xi$ we have
\begin{subequations}\label{conv+}
\begin{align}\label{conv-u+}  
&\baru_{\tau}\to u&&\!\!\text{weakly* in }L^\infty(0,T;\mathrm{BD}(\barOmega;\R^d)),
\\&\barpi_\tau\to\pi&&\!\!\text{weakly* in }L^\infty(0,T;\mathrm{Meas}(\barOmega;\R^{d\ti d}_{\mathrm{dev}}))\,\cap\,
\mathrm{BV}([0,T];L^1(\Omega;\R^{d\ti d}_{\mathrm{dev}})),
\\&\bare_{\mathrm{el},\tau}=e(\baru_{\tau}){-}\barpi_\tau\to 
e_{\mathrm{el}}=e(u){-}\pi
\hspace*{-18.9em}&&\hspace*{12em}\text{weakly* in }L^\infty(0,T;H^1(\Omega;\R^{d\times d}_\mathrm{sym})),\!\!
\\&\zeta_\tau\to\zeta&&\!\!\text{weakly* in }L^\infty(0,T;W^{1,r}(\Omega))\,\cap\,H^1(0,T;L^2(\Omega)),
\\\label{conv-r-Laplace}
&\mathrm{div}((1{+}\eps|\nabla\barzeta_\tau|^{r-2})\nabla\barzeta_\tau)
\to\mathrm{div}((1{+}\eps|\nabla\zeta|^{r-2})\nabla\zeta)\hspace*{-18em}&&\hspace*{18em}\text{weakly in }L^2(Q),
\\\label{conv-xi+}
&\barxi_\tau\to\xi&&\!\!\text{weakly in }L^2(Q);
\end{align}\end{subequations}
%\intertext{
actually, \eqref{conv-r-Laplace} uses also the maximal
monotonicity of the involved nonlinear operator.
%Moreover, for the already selected subsequence, we have also}
%\begin{subequations}\begin{align}
%\label{conv-nabla-zeta(T)}
%&&&\nabla\zeta_\tau(T)\to\nabla\zeta(T)&&\text{weakly in }\ L^r(\Omega;\R^d).
%\end{align}\end{subequations}
Moreover, by the BV-estimates and the Helly's selection principle, we can
also count with \eqref{conv-pi} and 
$\barzeta_\tau(t)\to\zeta(t)$ weakly in $L^2(\Omega)$,
and then by the a-priori $W^{1,r}$-estimate \eqref{est-Delta-zeta} 
also both the first and the second convergence in \eqref{conv-zeta};
both limits in \eqref{conv-zeta} are actually the same because 
the limit $\zeta$ is continuous in time into $L^2(\Omega)$ due to the
the a-priori $H^1$-estimate \eqref{est-Delta-zeta}.
%\COMMENT{CHECK WHETHER weakly SUFFICES!!!!}

By the compact embedding $W^{1,r}(\Omega)\Subset C(\barOmega)$ and 
by the Arzel\`a-Ascoli modification of the Aubin-Lions
theorem, cf.\ \cite[Lemma~7.10]{Roub13NPDE}, we have the 
compact embedding $C_\text{weak}([0,T];W^{1,r}(\Omega))\cap H^1(0,T;L^2(\Omega))
\Subset C([0,T];C(\barOmega))= C(\barQ)$.
Thus, from the estimate \eqref{est-Delta-zeta}, we obtain $\zeta_\tau\to\zeta$
in $C(\bar Q)$.
%and, in particular, also $\Indic_{S(\zeta_\tau)}^*(\DT\pi_\tau)\to\Indic_{S(\zeta)}^*(\DT\p)$
%weakly* in .
Further, we have
\begin{align}\label{est-of-diff-zeta}
&\big\|\underline\zeta_\tau-\zeta_\tau\big\|_{L^\infty(0,T;L^2(\Omega))}^2
=\sup_{0\le t\le T}\int_\Omega
\big|\underline\zeta_\tau(t,x)-\zeta_\tau(t,x)\big|\,\d x
\\&\nonumber
\qquad\qquad\qquad\le\int_\Omega\Big(\sup_{0\le t\le T}|\underline\zeta_\tau(t,x)-\zeta_\tau(t,x)|^2\Big)\d x
\\&\nonumber
\qquad\qquad\qquad=\int_\Omega\max_{k=1,...,T/\tau}\big|\zeta_\tau^k-\zeta_\tau^{k-1}\big|^2\d x
\le\int_\Omega\sum_{i=1}^{T/\tau}\big|\zeta_\tau^k-\zeta_\tau^{k-1}\big|^2\d x
\\&\qquad\qquad\qquad=\int_\Omega
\tau\sum_{i=1}^{T/\tau}\tau\Big|\frac{\zeta_\tau^k-\zeta_\tau^{k-1}}\tau\Big|^2\d x
=\tau\int_Q\!\big|\DT\zeta_\tau\big|^2\,\d x\d t.
\nonumber\end{align} 
%
%For any $x\!\in\!\Omega$, we have\begin{align}\nonumber
%\|\underline\zeta_\tau(\cdot,x)-\zeta_\tau(\cdot,x)\|_{L^\infty(0,T)}^2&=
%\sup_{0\le t\le T}\big|\underline\zeta_\tau(t,x)-\zeta_\tau(t,x)\big|^2
%=\max_{k=1,...,T/\tau}\big|\zeta_\tau^k-\zeta_\tau^{k-1}\big|^2
%\\&\le\sum_{i=1}^{T/\tau}\big|\zeta_\tau^k-\zeta_\tau^{k-1}\big|^2
%=\tau\sum_{i=1}^{T/\tau}\tau\Big|\frac{\zeta_\tau^k-\zeta_\tau^{k-1}}\tau\Big|^2
%=\int_0^T\!\!\tau\DT\zeta_\tau(t,x)^2\,\d t.\end{align} 
%Integrating it over $\Omega$, we obtain
%$\|\underline\zeta_\tau-\zeta_\tau\|_{L^\infty(0,T;L^2(\Omega))}\le C\sqrt\tau\|\DT\zeta_\tau\|_{L^2(0,T;L^2(\Omega))}$; here also the estimate
%$$\int_\Omega\Big(\sup_{0\le t\le T}|z(t,x)|^2\Big)\d x
%\ge\sup_{0\le t\le T}\int_\Omega|z(t,x)|^2\d x=\|z\|_{L^\infty(0,T;L^2(\Omega))}^2$$ 
%was used for $z=\underline\zeta_\tau-\zeta_\tau$.
Then, using the Gagliardo-Nirenberg inequality $\|z\|_{L^\infty(\Omega)}^{}\le 
C_\eps\|z\|_{L^2(\Omega)}^\eps\|z\|_{W^{1,r}(\Omega)}^{1-\eps}$
for some small $0<\eps<1$ depending on $r>d$, we can interpolate \eqref{est-of-diff-zeta}, i.e.\ 
$\|\underline\zeta_\tau-\zeta_\tau\|_{L^\infty(0,T;L^2(\Omega))}\le 
\sqrt\tau\|\DT\zeta_\tau\|_{L^2(Q)}$, with 
$\|\underline\zeta_\tau-\zeta_\tau\|_{L^\infty(0,T;W^{1,r}(\Omega))}\le C$
to obtain $\|\underline\zeta_\tau-\zeta_\tau\|_{L^\infty(Q)}\to0$.
Thus \eqref{conv-zeta-Loo} is proved.

\medskip

%\noindent \emph{Step 2: Strong convergence of $\bare_{\mathrm{el},\tau}$.} 
%By the positive-semidefiniteness of $\mathbb C$, we can estimate
%\begin{align}\nonumber
%&\COMMENT{a\ constant(p)??}\eps\big\|\bare_{\mathrm{el},\tau}-e_{\mathrm{el}}\big\|_{L^p(Q;\R^{d\times d})}^p
%\\&\nonumber\le
%\int_Q\mathbb C(\underline\zeta_\tau)\big(\bare_{\mathrm{el},\tau}-e_{\mathrm{el}}\big){:}
%\big(\bare_{\mathrm{el},\tau}-e_{\mathrm{el}}\big)
%%e(\baru_\tau-u)-\barpi_\tau+\pi\big){:}
%%\big(e(\baru_\tau-u)-\barpi_\tau+\pi\big)
%+\eps\big(|\bare_{\mathrm{el},\tau}|^{p-2}\bare_{\mathrm{el},\tau}-|e_{\mathrm{el}}|^{p-2}e_{\mathrm{el}}\big)
%{:}\big(\bare_{\mathrm{el},\tau}-e_{\mathrm{el}}\big)\dd x\d t
%\\&\nonumber=\int_Q-(\bar\chi_\tau+\mathbb C e_{\mathrm{el}})
%{:}\big(\bare_{\mathrm{el},\tau}-e_{\mathrm{el}}\big)
%%(\underline\zeta_\tau)\big(e(u)-\pi\big){:}
%%\big(e(\baru_\tau-u)-\barpi_\tau+\pi\big)
%\dd x\d t\end{align}
%where $\bar\chi_\tau\in N_{S(\underline\zeta_\tau)}^{}(\DT\pi_\tau)$. \COMMENT{HERE AV PROBLEM: $\bar\chi_\tau$ bounded in $L^\infty(Q;\R^{d\times d}_\mathrm{dev})$
%but $\bare_{\mathrm{el},\tau}\to e_{\mathrm{el}}$ only weakly in $L^2$! IF INDEED NOT POSSIBLE, THEN THE VISCOUS VARIANT SHOULD WORK - possibly with INERTIA
%but the energy conservation cannot be expected if $p>2$ (as needed).}

\medskip

\noindent \emph{Step 2: Energy inequality.} 
The convergence \eqref{conv+} allows already for passage
in the limit in the inequality \eqref{engr-bal-disc} by 
lower semicontinuity in the left-hand side and by continuity in the
right-hand side of \eqref{engr-bal-disc}.

The limit passage in $\scrE(T,u_\tau(T),\pi_\tau(T),\zeta_\tau(T))$ is by the convexity 
of $\scrE(T,\cdot,\cdot,\zeta)$
and the compactness in $\zeta$, while for $\int_0^T
\partial_t\scrE(t,\underline u_\tau(t),\underline\pi_\tau(t),\underline\zeta_\tau(t))\dd t$
we use the continuity of $\partial_t\scrE(t,\cdot,\cdot,\cdot)$ from \eqref{DTE} and 
the Lebesgue theorem; more in detail, we use the assumptions \eqref{ass-load}
and the weak convergence \eqref{conv-eel}.

The only remaining (and nontrivial) term is the dissipation 
$\widehat{\scrR}$-term. Let us note that, as 
the discrete flow rule 
$N_{S(\underline\zeta_\tau)}^{}(\DT\pi_\tau)
\ni\mathrm{dev}(\mathbb C (\underline\zeta_\tau)\bare_{\mathrm{el},\tau}
%+\eps|\bare_{\mathrm{el},\tau}|^{p-2}\bare_{\mathrm{el},\tau})
-\mathrm{div}\,\mathfrak{h}_\tau^k)$ as well as 
 the dissipation rate 
%$\Indic_{S(\underline\zeta_\tau)}^*(\DT\pi_\tau)$
$\sY(\underline\zeta_\tau)|\DT\pi_\tau|$ uses $\underline\zeta_\tau$
and not just $\zeta_\tau$, we needed to prove \eqref{conv-zeta-Loo} in 
Step 1.
%
%However, we need the uniform convergence not of $\zeta_\tau$ but of 
%$\underline\zeta_\tau$ which occurs in the discrete flow rule 
%$N_{S(\underline\zeta_\tau)}^{}(\DT\pi_\tau)
%\ni\mathrm{dev}(\mathbb C (\underline\zeta_\tau)\bare_{\mathrm{el},\tau}
%%+\eps|\bare_{\mathrm{el},\tau}|^{p-2}\bare_{\mathrm{el},\tau})
%-\mathrm{div}\,\mathfrak{h}_\tau^k)$
%and therefore, using the ansatz \eqref{ass-S}, also in the dissipation rate 
%%$\Indic_{S(\underline\zeta_\tau)}^*(\DT\pi_\tau)$
%$\sY(\underline\zeta_\tau)|\DT\pi_\tau|$
%which we need to converge in the sense \eqref{lsc-diss-rate} below.
%%However, we cannot directly use \cite[Lemma~7.10]{Roub13NPDE} because 
%%$\underline\zeta_\tau\not\in C_\mathrm{weak}([0,T];W^{1,r}(\Omega))$. 
%Instead, we need to estimate the difference 
%$\sY(\underline\zeta_\tau)|\DT\pi_\tau|-\sY(\zeta_\tau)|\DT\pi_\tau|$.
%To this goal, relying on uniform continuity of $\sY$ on 
%%bounded sets
%$[0,1]$, we need to prove also \eqref{conv-zeta-Loo}.
%%$\underline\zeta_\tau\to\zeta$ in $L^\infty(Q)$. 
%for which we use that, by \eqref{conv-zeta-Loo}, 
Therefore,
we have at disposal the estimate 
\begin{align}
\big\|(\sY(\underline\zeta_\tau)-\sY(\zeta))|\DT\pi_\tau|
\big\|_{\mathrm{Meas}(\barQ)}\le
%C_0
\ell_{\sY}^{}\big\|\underline\zeta_\tau-\zeta\big\|_{L^\infty(Q)}\big\|\DT\pi_\tau\big\|_{\mathrm{Meas}(\barQ)}
\to0
\end{align}
with 
%$C_0=\sup_{|A|\le1}\Indic_{B_1}^*(A)$ and 
$\ell_{\sY}^{}$ the modulus of 
Lipschitz continuity of $\sY$ on $[0,1]$, cf.\ the assumption \eqref{ass-S}.
Then, using also $\zeta_\tau\to\zeta$ in $C(\barQ)$ already proved, we obtain
\begin{align}\label{lsc-diss-rate}
&\liminf_{\tau\to0}\int_0^T\!\widehat\scrR(\underline\zeta_\tau;\DT\pi_\tau,\DT\zeta_\tau)\dd t
=\liminf_{\tau\to0}\int_\barQ \sY(\underline\zeta_\tau)\big|\DT\pi_\tau\big|(\d x\d t)
\\[-.3em]&\hspace{2em}\nonumber
=\lim_{\tau\to0}\int_\barQ\big(\sY(\underline\zeta_\tau)-\sY(\zeta)\big)
\big|\DT\pi_\tau\big|(\d x\d t)
+\liminf_{\tau\to0}\int_\barQ\!\sY(\zeta)\big|\DT\pi_\tau\big|(\d x\d t)
\\[-.3em]\nonumber
&\hspace{2em}
\ge\ \ 0\ +\int_\barQ\!\sY(\zeta)\big|\DT\pi\big|(\d x\d t);
\end{align}
for the used weak* lower semicontinuity of 
$\DT\pi\mapsto\int_\barQ\!\sY(\zeta)|\DT\pi|(\d x\d t)$ 
we refer e.g.\ to \cite{AmFuPa00FBVF,Gius03DMCV}.

\medskip

\noindent
\emph{Step 3: 
Limit passage in the semi-stability \eqref{semi-stab-disc} towards \eqref{semi-stab}.}
For any $(\wt u,\wt\pi)$ used in  \eqref{semi-stab-disc}, we have to
find at least one so-called mutual recovery sequence 
$\{(\widehat u_\tau,\widehat \pi_\tau)\}_{\tau>0}$
in the sense that
\begin{align*}\nonumber
\limsup_{\tau\to0}\scrE(t,\wt u_\tau,\wt\pi_\tau,\underline\zeta_\tau(t))
+\scrR(\underline\zeta_\tau(t);\wt\pi_\tau{-}\barpi_\tau(t),0)-
\scrE(t,\baru_\tau(t),\barpi_\tau(t),\underline\zeta_\tau(t))
\\\le
\scrE(t,\wt u,\wt\pi,\zeta(t))
+\scrR(\underline\zeta_\tau(t);\wt\pi{-}\pi(t),0)-
\scrE(t,u(t),\pi(t),\zeta(t)).
\end{align*}
We choose   
\begin{align}\label{mrs}
\wt u_\tau=\baru_\tau(t)+\wt u-u(t)\ \ \ \text{ and }\ \ \ 
\wt \pi_\tau=\bar \pi_\tau(t)+\wt \pi-\pi(t).
\end{align}
Then, by using the cancellation and the binomial formula of the type 
$a^2-b^2=(a{+}b)(a{-}b)$ here
in the form like $\mathbb C\wt e{:}\wt e-\mathbb Ce{:}e=\mathbb C(\wt e{+}e){:}(\wt e{-}e)$
and $\mathbb H\nabla\wt e{\Vdots}\nabla\wt e-\mathbb H\nabla e{\Vdots}\nabla e
=\mathbb H\nabla(\wt e{+}e){\Vdots}\nabla(\wt e{-}e)$, cf.\ \eqref{semi-stab+}, and 
by making the substitution \eqref{mrs}, we have 
\begin{align}\label{semi-stab-proof}
&\lim_{\tau\to0}\scrE(t,\wt u_\tau,\wt\pi_\tau,\underline\zeta_\tau(t))
+\scrR(\underline\zeta_\tau(t);\wt\pi_\tau{-}\barpi_\tau(t),0)-
\scrE(t,\baru_\tau(t),\barpi_\tau(t),\underline\zeta_\tau(t))
\\\nonumber&\quad=\lim_{\tau\to0}\bigg(\int_{\Omega}\,\frac12\mathbb C(\underline\zeta_\tau(t))\big(e(\wt u_\tau{+}\baru_\tau(t){+}2\uD(t)){-}\wt\pi_\tau{-}\barpi_\tau(t)\big)
\\\nonumber&\qquad\qquad\qquad\qquad\qquad\quad
:\big(e(\wt u_\tau{-}\baru_\tau(t)){-}\wt\pi_\tau{+}\barpi_\tau(t)\big)
-g(t){\cdot}(\wt u_\tau{-}\baru_\tau(t))
\\\nonumber&\qquad
+\frac12\mathbb H\nabla\big(e(\wt u_\tau{+}\baru_\tau(t){+}2\uD(t)){-}\wt\pi_\tau{-}\barpi_\tau(t)\big)
\\\nonumber&\qquad\qquad\qquad\qquad\qquad\quad
\Vdots
\nabla\big(e(\wt u_\tau{-}\baru_\tau(t)){-}\wt\pi_\tau{+}\barpi_\tau(t)\big)\dd x
\\\nonumber&\qquad
+\int_{\barOmega}\!\big[
%\Indic_{S(\underline\zeta_\tau(t))}^*\big(\wt\pi_\tau{-}\barpi_\tau(t)\big)
\sY(\underline\zeta_\tau(t))\big|\wt\pi_\tau{-}\barpi_\tau(t)\big|(\d x)
%-\int_{\Omega}\!g(t){\cdot}(\wt u_\tau{-}\baru_\tau(t))\dd x
-\!\int_{\GN}\!\!\frM(t){\cdot}(\wt u_\tau{-}\baru_\tau(t))\dd S\bigg)
\\\nonumber&\quad
=\lim_{\tau\to0}\bigg(\int_{\Omega}\,\frac12\mathbb C(\underline\zeta_\tau(t))
\big(e(\wt u_\tau{+}\baru_\tau(t){+}2\uD(t)){-}\wt\pi_\tau{-}\barpi_\tau(t)\big)
\\\nonumber&\qquad\qquad\qquad\qquad\qquad\quad
:\big(e(\wt u{-}u(t)){-}\wt\pi{+}\pi(t)\big)
-g(t){\cdot}(\wt u{-}u(t))
\\\nonumber&\qquad
+\frac12\mathbb H\nabla\big(e(\wt u_\tau{+}\baru_\tau(t){+}2\uD(t)){-}\wt\pi_\tau{-}\barpi_\tau(t)\big)\Vdots\nabla\big(e(\wt u{-}u(t)){-}\wt\pi{+}\pi(t)\big)\dd x
\\\nonumber&\qquad
+\int_{\barOmega}\!
%\big[\Indic_{S(\underline\zeta_\tau(t))}^*\big(\wt\pi{-}\pi(t)\big)
\sY(\underline\zeta_\tau(t))\big|\wt\pi{-}\pi(t)\big|(\d x)\bigg)
%-\int_{\Omega}g(t){\cdot}(\wt u{-}u(t))\dd x
-\!\int_{\GN}\!\!\frM(t){\cdot}(\wt u{-}u(t))\dd S
\\\nonumber&\quad
=\int_{\Omega}\,\frac12\mathbb C(\underline\zeta(t))
\big(e(\wt u{+}\baru(t){+}2\uD(t)){-}\wt\pi{-}\barpi(t)\big)
:\big(e(\wt u{-}u(t)){-}\wt\pi{+}\pi(t)\big)
\\\nonumber&\quad\quad
+\frac12\mathbb H\nabla\big(e(\wt u{+}\baru(t){+}2\uD(t)){-}\wt\pi{-}\barpi(t)\big)\Vdots\nabla\big(e(\wt u{-}u(t)){-}\wt\pi{+}\pi(t)\big)\dd x
\\\nonumber&\quad\quad
+\int_{\barOmega}\!\sY(\zeta)\big|\wt\pi{-}\pi(t)\big|
%\big[\Indic_{S(\zeta(t))}^*\big(\wt\pi{-}\pi(t)\big)\big]
(\d x)
-\int_{\Omega}\!g(t){\cdot}(\wt u{-}u(t))\dd x
-\!\int_{\GN}\!\!\frM(t){\cdot}(\wt u{-}u(t))\dd S
\\&\quad=\scrE(t,\wt u,\wt\pi,\zeta(t))+\scrR(\zeta(t);\wt\pi{-}\pi(t),0)-
\scrE(t,u(t),\pi(t),\zeta(t)).
\nonumber\end{align}
%\COMMENT{HONZO, TO ``$2\uD(t)$'' JE SNAD DOBRE NEBOT $\underline\zeta_\tau(t)$
%JE STEJNY TAK SE PROSTE POUZIJE BINOMICKA FORMULE} 
Note that we used also $\sY(\underline\zeta_\tau(t))|\wt\pi{-}\pi(t)|
\to\sY(\zeta)|\wt\pi{-}\pi(t)|$ in $\mathrm{Meas}(\barOmega)$
due to the continuity assumption \eqref{ass-S} on $\sY$ and due to the 
convergence $\underline\zeta_\tau(t)\to\zeta(t)$ in $C(\bar\Omega)$
which follows from the second estimates in \eqref{conv-zeta} and the 
compact embedding $W^{1,r}(\Omega)\subset C(\bar\Omega)$.

\medskip

\noindent
\emph{Step 4: Limit passage in the damage flow rule \eqref{flow-rule-disc} 
towards \eqref{flow-rule-weak}.}
We need to prove that $\bare_{\mathrm{el},\tau}\to e_{\mathrm{el}}$ strongly in 
$L^2(Q;\R_{\rm sym}^{d\times d})$. To this goal, we first realize that 
$\nabla\bare_{\mathrm{el},\tau}(t)\to\nabla e_{\mathrm{el}}(t)$
weakly in $L^2(\Omega;\R^{d\times d\times d})$ as pronounced in \eqref{conv-eel};
here we use the uniqueness of stresses (counting the already selected 
subsequence \eqref{conv+} and its limit), cf.~the arguments in 
\cite[Thm.5.9]{DaDeMo06QEPL} or also in \cite[Sect.4.2.3]{Maug92TPF}
for simple materials without damage. Here, using also absolute continuity 
valid due to viscosity in damage flow rule
%\COMMENT{important: absolute continuity valid due to viscosity in damage flow rule and 
%the argumentation is to be used for the hyperstresses which are not explicitly 
%subjected to damage: from} $\langle\sigma_{\rm el}^{(1)}-\sigma_{\rm el}^{(2)}),
%\DT e_{\rm el}^{(1)}-\DT e_{\rm el}^{(2)}\rangle\le0$ ????? 
we obtain
\begin{align}\label{uniqueness-of-stress}
&\frac12\frac{\d}{\d t}\Big(\big\langle\mathbb{H}\nabla(e_{\rm el}^{(1)}{-}e_{\rm el}^{(2)}),\nabla(e_{\rm el}^{(1)}{-}e_{\rm el}^{(2)})\big\rangle
+\big\langle\mathbb{C}(\zeta)(e_{\rm el}^{(1)}{-}e_{\rm el}^{(2)}),e_{\rm el}^{(1)}{-}e_{\rm el}^{(2)}\big\rangle\Big)
\\&\nonumber
%\ \ \ \ \ 
\qquad\qquad\qquad\qquad=-\frac12\big\langle\mathbb{C}'(\zeta)\DT\zeta(e_{\rm el}^{(1)}{-}e_{\rm el}^{(2)}),
e_{\rm el}^{(1)}{-}e_{\rm el}^{(2)}\big\rangle
\\&\nonumber
\qquad\qquad\qquad\qquad\le\max_{0\le z\le1}|\mathbb{C}'(z)|\big\|\DT\zeta\big\|_{L^2(\Omega)}
\big\|e_{\rm el}^{(1)}{-}e_{\rm el}^{(2)}\big\|_{L^4(\Omega;\R^{d\times d})}^2.
\end{align}
Note that, for $\mathbb{H}=0$ and $\mathbb{C}'=0$, it reduces to the simple 
inequality $\langle\sigma_{\rm el}^{(1)}-\sigma_{\rm el}^{(2)}),
\DT e_{\rm el}^{(1)}-\DT e_{\rm el}^{(2)}\rangle\le0$ used in 
\cite{DaDeMo06QEPL,Maug92TPF}. Here, we should integrate 
\eq{uniqueness-of-stress} over $[0,t]$, use positive-definiteness of 
$\mathbb{H}$ and $\mathbb{C}(\cdot)$, and eventually Gronwall's inequality,
%
%$\frac12\frac{\d}{\d t}\langle\mathbb{H}\nabla(e_{\rm el}^{(1)}{-}e_{\rm el}^{(2)}),\nabla(e_{\rm el}^{(1)}{-}e_{\rm el}^{(2)})\rangle=
%-\langle\mathbb{C}(\zeta)(e_{\rm el}^{(1)}{-}e_{\rm el}^{(2)}),\DT e_{\rm el}^{(1)}{-}\DT e_{\rm el}^{(2)}\rangle=-\frac12\frac{\d}{\d t}\langle\mathbb{C}(\zeta)(e_{\rm el}^{(1)}{-}e_{\rm el}^{(2)}),e_{\rm el}^{(1)}{-}e_{\rm el}^{(2)}\rangle
%+\frac12\langle\mathbb{C}'(\zeta)\DT\zeta(e_{\rm el}^{(1)}{-}e_{\rm el}^{(2)}),e_{\rm el}^{(1)}{-}e_{\rm el}^{(2)}\rangle$ where the last term is to be handled by 
%Gronwall's inequality 
%...... $\le \|\DT\zeta\|_{L^2(\Omega)}
%\|e_{\rm el}^{(1)}{-}e_{\rm el}^{(2)}\|_{L^4(\Omega;\R^{d\times d})}^2$ 
which works here certainly even for $d\le4$ for which the embedding 
$H^2(\Omega)\subset W^{1,4}(\Omega)$ holds. By this way, we obtain
$e_{\rm el}^{(1)}=e_{\rm el}^{(2)}$.
%\COMMENT{DETAILS!}.
Thus, using the compact embedding, we also know that 
$\bare_{\mathrm{el},\tau}(t)\to e_{\mathrm{el}}(t)$ strongly in 
 $L^{6-\epsilon}(\Omega;\R_{\rm sym}^{d\times d})$ if $d\le3$.
Then, by the uniform bounds in time and by Lebesgue's theorem
used e.g.\ to $t\mapsto\|\bare_{\mathrm{el},\tau}(t)
-e_{\mathrm{el}}(t)\|_{L^1(\Omega;\R_{\rm sym}^{d\times d})}$, we can see that
$\bare_{\mathrm{el},\tau}\to e_{\mathrm{el}}$ strongly even in 
$L^{1/\epsilon}(0,T;L^{6-\epsilon}(\Omega;\R_{\rm sym}^{d\times d}))$
with each small $\epsilon>0$.

%we use uniform monotonicity of 
%$(u,\pi)\mapsto(-\mathrm{div}(\mathbb{C}(\zeta)e_{\rm el}
%-\mathrm{div}\mathbb{H}\nabla e_{\rm el})$ to estimate
%\begin{align}\nonumber
%\eps\big\|\bare_{{\rm el},\tau}{-}e_{\rm el}\big\|_{L^2(Q;\R_{\rm sym}^{d\times d})}^2
%&\le\int_Q(\mathbb{C}(\underline\zeta_\tau)(\bare_{{\rm el},\tau}{-}e_{\rm el}):(\bare_{{\rm el},\tau}{-}e_{\rm el})+\mathbb{H}\nabla(\bare_{{\rm el},\tau}{-}e_{\rm el})\Vdots
%(\bare_{{\rm el},\tau}{-}e_{\rm el})\,\d x\d t\\\nonumber&\le\int_Q 
%\bar\sigma_\tau:(\barpi_\tau{-}\pi)+\bar g_\tau\cdot(\baru_\tau{-}u)
%-(\mathbb{C}(\underline\zeta_\tau)e_{\rm el}:(\bare_{{\rm el},\tau}{-}e_{\rm el})
%\\&\ \ -\mathbb{H}\nabla e_{\rm el}\Vdots
%(\bare_{{\rm el},\tau}{-}e_{\rm el})\,\d x\d t
%+\int_{\Snew}\bar f_\tau\cdot(\baru_\tau{-}u)\,\d S\d t\,\to\,0
%\label{e-strong}\end{align}
%where $\bar\sigma_\tau\in N_{S(\underline\zeta_\tau)}(\DT\pi_\tau)$ \COMMENT{ESTIMATE??? RATHER $\bar\sigma_\tau={\rm dev}...$??}
%and where we used a test \eqref{semi-stab-disc} by ...............

%Having already proved this strong convergence, of $\bare_{\mathrm{el},\tau}$, 
Then the only difficult remaining terms are $\kappa\int_Q
\mathrm{div}((1{+}\eps|\nabla\barzeta_\tau|^{r-2})\nabla\barzeta_\tau)
\DT\zeta_\tau\dd x\d t$ and $\int_Q\barxi_\tau(-\DT\zeta_\tau)\,\d x\d t$
because so far we know only the weak convergence  of $\DT\zeta_\tau$,  of 
$\mathrm{div}((1{+}\eps|\nabla\barzeta_\tau|^{r-2})\nabla\barzeta_\tau)$, and 
of $\barxi_\tau$ in $L^2(Q)$. We indeed cannot expect the 
limit, but we can proceed the following estimate:
\begin{align}\label{limsup}
&\limsup_{\tau\to0}\int_Q
\mathrm{div}((1{+}\eps|\nabla\barzeta_\tau|^{r-2})\nabla\barzeta_\tau)\DT\zeta_\tau\dd x\d t
\\[-.3em]\nonumber&\qquad=-\liminf_{\tau\to0}\int_Q(1{+}\eps|\nabla\barzeta_\tau|^{r-2})\nabla\barzeta_\tau{\cdot}
\nabla\DT\zeta_\tau\dd x\d t
\\\nonumber&\qquad\le\limsup_{\tau\to0}\int_\Omega\frac12\big|\nabla\zeta_0|^2+\frac\eps r\big|\nabla\zeta_0\big|^r
-\frac12\big|\nabla\zeta_\tau(T)\big|^2-\frac\eps r\big|\nabla\zeta_\tau(T)\big|^r
\dd x
\\
\nonumber
&\qquad\le\int_\Omega\frac12\big|\nabla\zeta_0|^2+\frac\eps r\big|\nabla\zeta_0\big|^r
-\frac12\big|\nabla\zeta(T)\big|^2-\frac\eps r\big|\nabla\zeta(T)\big|^r\dd x
\\\nonumber&\qquad
=\int_Q\mathrm{div}((1{+}\eps|\nabla\zeta|^{r-2})\nabla\zeta)\DT\zeta\dd x\d t
\end{align}
where we used 
%\eqref{conv-nabla-zeta(T)}
\eqref{conv-zeta} at $t=T$ and where the last equality 
relies on the regularity property 
$\mathrm{div}((1{+}\eps|\nabla\zeta|^{r-2})\nabla\zeta)\in L^2(Q)$
and can be proved either by a mollification in space \cite[Formula (3.69)]{PoRoTo10TCTF}
and or in time by a time-difference technique \cite[Formula (2.15)]{Gru95DPEF}.
%\COMMENT{also \cite{Roub13NPDE}}

The convergence in the inclusion $\barxi_\tau\in N_{[0,1]}(\barzeta_\tau)$ is 
easy due to the maximal monotonicity of $N_{[0,1]}(\cdot)$ and the convergences 
\eqref{conv-xi} and $\barzeta_\tau\to\zeta$ strongly in $L^2(Q)$ which can be 
proved by a generalized version of the Aubin-Lions theorem, cf.\ 
\cite[Corollary~7.9]{Roub13NPDE}, or here even in $L^\infty(Q)$ was proved as 
in Step 1. Having proved $\xi\in N_{[0,1]}(\zeta)$, we can also see that
\begin{align}\label{limit-in-xi-dot-zeta}
&\limsup_{\tau\to0}\int_Q\barxi_\tau(-\DT\zeta_\tau)\,\d x\d t
=\limsup_{\tau\to0}\bigg(\int_\Omega\Indic^{}_{[0,1]}(\zeta_0)\,\d x
-\int_\Omega\Indic^{}_{[0,1]}(\zeta_\tau(T))\,\d x\bigg)
\\[-.2em]&\nonumber\qquad\qquad\qquad\qquad
\le\int_\Omega\Indic^{}_{[0,1]}(\zeta_0)\,\d x-\int_\Omega\Indic^{}_{[0,1]}(\zeta(T))\,\d x
=\int_Q\xi(-\DT\zeta)\,\d x\d t,
\end{align}
which is needed for the limit passage in \eqref{flow-rule-disc}; in fact, even 
the limit and the equality hold in \eqref{limit-in-xi-dot-zeta}.

\medskip

\noindent \emph{Step 5: Energy equality.} 
We test \eqref{plast-dam13} which holds a.e.\ on $Q$ by $\DT\zeta$. This test 
is legal as all terms in \eqref{plast-dam13} as well as $\DT\zeta$ are in 
$L^2(Q)$. We again use the last equality in \eqref{limsup}. Moreover, as 
$\xi\in\partial\Indic_{[0,1]}(\zeta)$, we have $\int_Q\xi\DT\zeta\dd x\d t
=\int_\Omega\Indic_{[0,1]}(\zeta(T))-\Indic_{[0,1]}(\zeta(0))\d x=0-0=0$. We 
thus obtain
\begin{align}\label{engr-bal-dam}
&\int_\Omega
\frac\kappa2\big|\nabla\zeta(T)\big|^2
+\frac{\eps\kappa}r\big|\nabla\zeta(T)\big|^r-b(\zeta(T))\dd x
\ \
\\[-.3em]&\ \ +\!\int_Q\frac12\mathbb C'(\zeta)e_\mathrm{el}:e_\mathrm{el}
+\widehat a(\DT\zeta)\dd x\d t=\int_\Omega
\frac\kappa2\big|\nabla\zeta_0|^2+\frac{\eps\kappa}r\big|\nabla\zeta_0\big|^r-b(\zeta_0)
\dd x.
\nonumber\end{align}

Furthermore, we test formally \eqref{plast-dam1} by $\DT u$ and 
\eqref{plast-dam12} by $\DT\pi$. The rigorous calculations uses the
approximation of the 
%Lebesque
Stieltjes-type integral by Riemann sums and semistability, cf.\ 
\cite[Formulas (76)--(82)]{Roub13TPP} which adapts technique
developed in the theory of rate-independent processes
\cite{DaFrTo05QCGN,Miel05ERIS}. Here, as $\mathbb C$ is not constant, 
we will still see the term 
$(\frac12\mathbb C'(\zeta)e_\mathrm{el}{:}e_\mathrm{el})\DT\zeta$ which results 
by the formal substitution
$\mathbb C(\zeta)e_\mathrm{el}{:}\DT e_\mathrm{el}=\frac{\partial}{\partial t}
\frac12\mathbb C(\zeta)e_\mathrm{el}{:}e_\mathrm{el}
-(\frac12\mathbb C'(\zeta)e_\mathrm{el}{:}e_\mathrm{el})\DT\zeta$;
note that $\mathbb C(\zeta)e_\mathrm{el}{:}\DT e_\mathrm{el}$
is not well defined since $\DT e_\mathrm{el}$ is not well controlled.
Thus we obtain
\begin{align}\label{engr-bal-plast}
&\int_\Omega\frac12\mathbb C(\zeta(T))e_\mathrm{el}(T){:}e_\mathrm{el}(T)
+\frac12\mathbb H\nabla e_\mathrm{el}(T)\Vdots\nabla e_\mathrm{el}(T)\dd x
\\[-.3em]&\nonumber\qquad\quad+\int_{[0,T]\times\barOmega}\!\sY(\zeta)\big|\DT\pi\big|(\d x\d t)
=\int_Q\Big(\frac12\mathbb C'(\zeta)e_\mathrm{el}{:}e_\mathrm{el}\Big)
\DT\zeta\dd x\d t
\\[-.3em]&\qquad\qquad\qquad+\int_\Omega\frac12\mathbb C(\zeta_0)e_\mathrm{el}(0){:}e_\mathrm{el}(0)
+\frac12\mathbb H\nabla e_\mathrm{el}(0)\Vdots\nabla e_\mathrm{el}(0)\dd x.
\nonumber\end{align}
Summing \eqref{engr-bal-dam} and \eqref{engr-bal-plast} then gives the
energy balance \eqref{engr-bal}.
%\end{remark}
\end{proof}

\medskip

%\section{Discretization and implementation}\label{sec-comp}

Further, to implement the model computationally, we need to make 
a spatial discretisation of the time-discrete scheme 
\eqref{plast-dam-disc}--\eqref{BC-disc}. To this goal, we use
the lowest-order 
%spatial discretisation by the 
conformal finite-element method (FEM). In view of the used regularity 
\eqref{est-Delta-zeta+}, the straightforward discretisation 
therefore employs P2-elements for $u$ and $\zeta$ and P1-elements for $\pi$. 
Rigorously speaking, due to the assumed smoothness \eqref{ass-Omega},
one should consider FEM on a nonpolyhedral, curved domain. 
The minimization problems \eqref{minimization} are then to be restricted on 
the corresponding finite-dimensional subspaces, and the solution thus 
obtained is denoted by $u_{\tau h}^k$, $\pi_{\tau h}^k$, and $\zeta_{\tau h}^k$,
with $h>0$ denoting the mesh size. By this way, we obtain also the
piecewise constant and affine interpolants in time, denoted 
by $\baru_{\tau h}$ and $u_{\tau h}$, $\barpi_{\tau h}$ and $\pi_{\tau h}$, 
and eventually $\barzeta_{\tau h}$ and $\zeta_{\tau h}$.
Also, $\barxi_{\tau h}$ can be obtained analogously as before in 
Lemma~\ref{lem-disc-2}. 

\medskip

\begin{proposition}[Convergence of the FEM discretisation]\label{prop-FEM}
Let \eqref{ass} be satisfied, 
%$\Omega$ be polyhedral,
and the P2-FEM for $u$ and $\zeta$ and P1-FEM for $\pi$ is applied with $h>0$ 
the mesh size. Then:\\
\Item{(i)}{the a-priori estimates \eqref{est} and \eqref{est-of-xi} hold when 
modified for $u_{\tau h}$, $\pi_{\tau h}$, $\zeta_{\tau h}$, and $\barxi_{\tau h}$ 
with $C$ independent of $\tau>0$ and now of $h>0$, too.} 
\Item{(ii)}{Moreover, these fully discrete solutions converge (in terms of 
subsequences) in the mode as \eqref{conv} towards weak solutions according 
Definition~\ref{weak-sln} when simultaneously $\tau\to0$ and $h\to0$.}
\end{proposition}

%\medskip

The modification of the proof of this joint convergence of time-and-space 
discretisation is rather routine, the explicit construction of the mutual 
recovery sequence \eqref{mrs} taking additionally a finite-element 
approximation like in \cite{BaMiRo12QSSP}, namely
$\wt u_{\tau h}=\baru_{\tau h}(t)+\Pi_h^{(2)}(\wt u-u(t))$ and 
$\wt \pi_{\tau h}=\bar \pi_{\tau h}(t)+\Pi_h^{(1)}(\wt \pi-\pi(t))$
with $\Pi_h^{(1)}$ and $\Pi_h^{(2)}$ denoting a projector onto the P1- 
and P2 FE-spaces, respectively; we omit details about this modification.

\medskip

\begin{remark}[{\sl Damage discretised by P1-elements}] \label{remark:P1damage}
\upshape
The damage flow rule \eqref{plast-dam13} itself suggests to use only 
P1-elements for $\zeta$ which are, naturally, more easy to implement than the 
P2-elements used in Proposition~\ref{prop-FEM}. Then however 
\eqref{est-Delta-zeta+} cannot be expected for the FEM approximation of 
$\zeta$ and also a direct P-1 FEM analog of \eqref{flow-rule-disc} 
cannot hold. Instead of \eqref{flow-rule-disc}, we have
\begin{align}\label{flow-rule-disc+}
&\int_Q\!\bigg(a(v)+\Big(\frac12\mathbb C'(\underline\zeta_{\tau h})
%(e(u){-}\pi):(e(u){-}\pi)
\bare_\mathrm{el,\tau}:\bare_\mathrm{el,\tau}
-b'(\barzeta_{\tau h})+\barxi_{\tau h}\Big)(v-\DT\zeta_{\tau h})
\\[-.4em]
&\hspace{4em}
%\Delta\zeta
+\kappa\,
\big((1{+}\eps|\nabla\barzeta_{\tau h}|^{r-2})\nabla\barzeta_{\tau h}\big)
\cdot\nabla(v{-}\DT\zeta_{\tau h})
%+\kappa\big((1+\eps|\nabla\zeta|^{r-2})\nabla\zeta\big)\cdot\nabla(v-\DT\zeta)
\bigg)\dd x\dd t
\ge\int_Qa(\DT\zeta_{\tau h})\dd x\dd t
\nonumber\end{align}
for any $v$ valued in the finite-dimensional P1-FE subspace.
%\COMMENT{CHECK HOW $\barxi_{\tau h}$ WORKS!!}
Yet, the sequence $\{\nabla\DT\zeta_{\tau h}\}_{\tau>0,h>0}$ cannot be expected 
bounded. Thus, for the limit passage, instead of \eqref{flow-rule-disc+} one 
should rather use the discrete by-part integration (summation) in time like 
we did in \eqref{limsup}, leading to
\begin{align}\label{flow-rule-disc++}
&\int_Q\bigg(a(v)+\Big(\frac12\mathbb C'(\underline\zeta_{\tau h})
%(e(u){-}\pi):(e(u){-}\pi)
\bare_\mathrm{el,\tau}:\bare_\mathrm{el,\tau}
-b'(\barzeta_{\tau h})+\barxi_{\tau h}\Big)(v-\DT\zeta_{\tau h})
%\Delta\zeta
\\[-.3em]\nonumber
&\hspace{2em}+\kappa\,
\big((1{+}\eps|\nabla\barzeta_{\tau h}|^{r-2})\nabla\barzeta_{\tau h}\big)
\cdot\nabla v\bigg)\dd x\dd t
+\int_\Omega\frac\kappa2|\nabla\zeta_0|^2\!+\frac{\eps\kappa}r|\nabla\zeta_0|^r\dd x
%+\int_\O|\zeta_{\tau h}(0)|^2+\frac\eps r|\zeta_{\tau h}(0)|^r\dd x
\\[-.3em]\nonumber
&\hspace{4em}
\ge\int_Q\!a(\DT\zeta_{\tau h})\dd x\dd t
+\int_\Omega\frac\kappa2|\nabla\zeta_{\tau h}(T)|^2\!
+\frac{\eps\kappa}r|\nabla\zeta_{\tau h}(T)|^r\dd x
\end{align}
which holds for any $v$ valued in the P1-finite-element space. Now, however, 
we do not have the estimates \eqref{est-Delta-zeta+} and \eqref{est-of-xi}.
Anyhow, the limit passage seems possible by using the strategy proposed by Colli and 
Visintin \cite{ColVis90CDNE}, cf.\ also \cite[Sect.\,11.1.2]{Roub13NPDE}, 
allowing for the stored energy $\scrE$ taking values $+\infty$ but relying on 
boundedness of $\scrR$, as indeed our situation. The convergence is, of course,
in a weaker mode than \eqref{conv}.
%using also 
%that $\zeta_{\tau h}(T)\to\zeta(T)$ weakly in $W^{1,r}(\Omega)$ by the 
%arguments like we used for \eqref{conv-nabla-zeta(T)}. 
Only after this limit passage, we can prove the regularity \eqref{H2-zeta}
and go back to the weak formulation \eqref{flow-rule-weak}
by using also the arguments which we use for the last equality in \eqref{limsup}.
\end{remark}

\section{Implementation of the fully discrete model}\label{sec-comp}
%%        ~~~~~~~~~~~~~~~~~~~~~~~~~~~~~~~~~~~~~~~~~
%
%.... shortcut $\eps=0$ and $\mathbb H=0$... 
%For $d=2$, 
%
%The mentioned simplification  $\mathbb H=0$ allows us to use more simply
%P1 elements for $u$ and $\zeta$ and P0 elements for $\pi$.
%
%.....
%%$C(\cdot)$ affine, 
%piecewise quadratic, 
%$b(\cdot)$ affine,
%
%....transformation to recursive alternating quadratic-programming 
%problem
%
%....small isotropic hardening \COMMENT{SHALL WE USE IT???} to facilitate 
%efficient usage of quasi-Newton iterative method \COMMENT{???} 
%\\
%\\

The implementation of the model addressed in Proposition~\ref{prop-FEM}
is rather cumbersome because of high-order FEM involved. 
Therefore we dare make few shortcuts:
%following shortcuts are taken to account: 
P1-elements are used for damage $\zeta$ according to Remark \ref{remark:P1damage}. 
Moreover, the (anyhow usual small and even not reliably known) hyperelasticity moduli 
are neglected, i.e.\ $\mathbb H=0$ and then small-strain tensor gradients 
$\nabla e(u)$ are not involved. Consequently, only P1-elements 
can be used for displacement $u$ and P0-elements for plastic strain. 
Only the case $d=2$ is treated, so the previous analytical part have 
required $r>2$ and we dare 
%can alternatively we 
make another (indeed small) shortcut by considering $r=2$ (and therefore
by putting $\epsilon=0$ the damage-gradient term in \eqref{seismic-E+} become
quadratic). 
%Then the regularization of the gradient term of $\zeta$ in \eqref{seismic-E+} disappears. \\

The material is assumed isotropic with properties linearly dependent on damage.
The isotropic elasticity tensor is assumed as 
\begin{equation}\label{C-ass-implem}
\mathbb C_{ijkl} (\zeta):=[(\lambda_1{-}\lambda_0) \zeta + \lambda_0] \delta_{ij}\delta_{kl}  + [(\mu_1{-}\mu_0) \zeta + \mu_0] (\delta_{ik}\delta_{jl}+  \delta_{il}\delta_{jk})               
\end{equation}
where $\lambda_1,\mu_1$ and  $\lambda_0,\mu_0$ are two sets of Lam\'e parameters satisfying
$$
\lambda_1 \geq \lambda_0 \geq 0, \qquad \mu_1 \geq \mu_0 > 0.
$$ 
Here, $\delta$ denotes the Kronecker 
%delta tensor. 
symbol. This choice implies that the elastic-moduli tensor satisfies \eqref{ass-C} and 
it is even positive-definite-valued (and therefore invertible). Values of 
%constants 
$\mathbb{C}^{}_{\mbox{\tiny\rm D}}(\zeta)$ and $c^{}_{\mbox{\tiny\rm S}}(\zeta)$ in \eqref{ass-C} 
follow from a decomposition of the elastic strain energy $\frac12\mathbb C(\cdot)e{:}e$ 
into the deviatoric and the volumetric parts of the strain tensor $e$. The stored 
energy of damage compliant with \eqref{ass-b} is assumed in the form
\begin{equation}
b(\zeta):=b_1\, \zeta,
\end{equation}
where 
%$c>0$ is given (material) parameter, so it fulfills \eqref{ass-b}. 
%Actually, 
$b_1>0$ means the specific energy stored in the microcracks/microvoids created by
damaging the material. By healing, this energy can be recovered back. 
The plastic yield stress compliant with \eqref{ass-S} is assumed in the form 
\begin{equation}\label{special-yield-stress}
\sY(\zeta)=\big(\sigma_{\mbox{\tiny\rm Y},1}{-}\sigma_{\mbox{\tiny\rm Y},0}\big)\zeta+\sigma_{\mbox{\tiny\rm Y},0},
\end{equation}
where  $\sigma_{\mbox{\tiny\rm Y},1}\ge\sigma_{\mbox{\tiny\rm Y},0} >0$.
% is given (material) parameter, so it fulfills \eqref{ass-S}. 
The damage-dissipation potential is assumed in the piecewise quadratic form
\begin{equation}\label{damage_disipation_potential}
a(\DT\zeta):=\frac12a_1(\DT\zeta^+)^2+\frac12a_2(\DT\zeta^-)^2 +a_3(\DT\zeta^-),
\end{equation} 
where 
$\DT\zeta^+=\max \{0,\DT\zeta\}$ and $\DT\zeta^-=\max \{-\DT\zeta, 0 \}$ and $a_1, a_2, a_3$ are given (material) nonnegative parameters. Values of $a_1$ and $a_2$ determine rate-dependent parts of healing and damage model components and the value of $a_3$ a rate-independent damage activation. The form of $a(\cdot)$ satisfies \eqref{ass-a}. 
%\\ 
%The higher the values of $a_1$ and $a_2$ contributed to slower healing and damage processes.
%We consider the case of two space-dimensions $d=2$ only and decompose the domain $\Omega$ in a regular triangulation $\mathcal T$. We use P1-elements for displacement $u$ and %damage $\zeta$. The plastic strain $\pi$ is discretized in the space of P0-elements satisfying the element-wise trace-free condition $\tr \pi=0$. 

With respect to the fractional-step strategy of Section 3, we solve first 
for $(u_{\tau h}^k,\pi_{\tau h}^k)$ 
%=(u,\pi)(k\tau)$ 
from the elastoplastic minimization problems \eqref{minimization-1} and then $\zeta_\tau^k$ 
%=\zeta(k\tau)$ 
from the damage minimization problem \eqref{minimization-2} recursively for $k=1,...,T/\tau$.
% Thus  For given $(\zeta_{\tau h}^{k-1}, u_\tau^{k-1},\pi_\tau^{k-1})$ in the previous time step $k-1$. 
In view of the above shorcuts and simplifications, the minimization problems 
\eqref{minimization-1} and \eqref{minimization-2}  rewrite as
\begin{align}\label{minimization-1-implem}
&(u_{\tau h}^k,\pi_{\tau h}^k) = \argmin_{u,\pi}
 \int_{\Omega}\bigg(\frac12\mathbb C(\zeta_{\tau h}^{k-1})\big(e(u{+}\uDtauhk){-}\pi\big):\big(e(u{+}\uDtauhk){-}\pi \big)
   \\[-.3em]  & \qquad \qquad \qquad \qquad\qquad 
% \qquad \qquad \qquad \qquad
-g_{\tau h}^k{\cdot}u+ \sY(\zeta_{\tau h}^{k-1}) | \pi{-}\pi_{\tau h}^{k-1}| \bigg)\dd x
   -\!\int_{\GN}\!\! f_{\tau h}^k {\cdot}u\dd S,
\nonumber\\
&\zeta_{\tau h}^k   = \argmin_{\zeta}   \int_{\Omega}\bigg(\frac12\mathbb C(\zeta)\big(e(u_{\tau h}^k{+}\uDtauhk){-}\pi_{\tau h}^k\big):\big(e(u_{\tau h}^k{+}\uDtauhk{-}\pi_{\tau h}^k\big){-}b_1\zeta  
\label{minimization-2-implem} \\[-.3em] &  \qquad 
%\qquad \qquad \qquad \qquad
+\frac12 \kappa |\nabla\zeta|^2  + 
\frac{1}{2\tau} a_1( \zeta{-} \zeta_{\tau h}^{k-1})^+ +
\frac{1}{2\tau} a_2( \zeta{-} \zeta_{\tau h}^{k-1})^-  +
a_3( \zeta{-} \zeta_{\tau h}^{k-1})^-
\bigg)\dd x ,  \nonumber 
\end{align}
where $u$ is searched over P1-elements satisfying Dirichlet boundary 
conditions, $\pi$ over P0-elements satisfying elementwise trace-free 
condition $\tr\pi=0$ and $\zeta$ over P1-elements satisfying the nodal box 
constraint $\zeta\in[0,1]$. The form of \eqref{minimization-1-implem} 
corresponds to the minimization problem of perfect plasticity with the 
elasticity tensor and  the plastic yield stress 
depending on the damage variable in the previous time level. The energy in 
\eqref{minimization-1-implem} is transformed to an energy in the variable $u$ 
only by substituting the elementwise dependency of $\pi$ on $u$, 
see \cite{AlCaZa99ANAPEH,CeKoSyVa14TDDSEP} for more details. Then, the 
quasi-Newton iterative methods is applied to solve $u_{\tau h}^k$ while 
$\pi_{\tau h}^k$ is reconstructed from it. More details on this specific 
elastoplasticity solver can be found e.g.\ in 
\cite{CeKoSyVa14TDDSEP,GKNT10ANAT,GruVal09SOTSPESNM}.

The damage minimization problem \eqref{minimization-2-implem} represents a 
minimization of a nonsmooth but strictly convex functional. It can be 
reformulated to a modified problem
\def\ZZ{z}
\begin{subequations}\begin{align}
\label{minimization-2-QP}
%(\zeta, \ZZ_+, \ZZ_-)= 
&\argmin_{\zeta, \ZZ_+,  \ZZ_-}\int_{\Omega}\bigg(\,\frac12
\mathbb C (\zeta) \big(e(u_{\tau h}^k{+}\uDtauhk){-}\pi_{\tau h}^k\big):\big(e(u_{\tau h}^k{+}\uDtauhk{-}\pi_{\tau h}^k\big)
\\[-.6em]\nonumber
&\qquad\qquad\quad-b_1\zeta+\frac12 \kappa |\nabla \zeta|^2   
+ \frac{1}{2 \tau} a_1  (\ZZ_+)^2  + \frac{1}{2 \tau} a_2  ( \ZZ_-)^2 + a_3  \ZZ_- 
\bigg)\dd x,
\\[-.1em]\label{minimization-2-QP+}
&\text{where }\
\ZZ_+ = (\zeta{-}\zeta_{\tau h}^{k-1})^+ , \ZZ_- = (\zeta{-}\zeta_{\tau h}^{k-1})^- 
\end{align}\end{subequations}
are additional `update' variables.
%\begin{equation}
%$\ZZ_+ = (\zeta{-}\zeta_{\tau h}^{k-1})^+ , \ZZ_- = (\zeta{-}\zeta_{\tau h}^{k-1})^-   $
%\end{equation}
%are introduced. 
It should be noted that $\zeta$ and $\zeta_{\tau h}^{k-1}$ are 
P1-functions and therefore $\ZZ_+$ and $\ZZ_-$ are not P1-functions in 
general on elements where nodal values of $\zeta{-}\zeta_{\tau h}^{k-1}$ 
%attain both 
alternate
signs. However, if we restrict $z_+, z_-$  to P1-functions while \eqref{minimization-2-QP+}
is required on at nodal points, then \eqref{minimization-2-QP} actually represents a 
%classical 
conventional quadratic-programming problem (QP), in which we require a linear and box constraints
\begin{equation}\label{linear_and_box_constraints_QP}
\zeta = \zeta_{\tau h}^{k-1} + \ZZ_+  - \ZZ_-, \qquad \ZZ_+ \in [0, 1- \zeta_{\tau h}^{k-1}], \qquad \ZZ_- \in [0, \zeta_{\tau h}^{k-1}].
\end{equation}
A quadratic cost functional of this QP problem has a positive-semidefinite Jacobian, 
since there are no Dirichlet boundary conditions on the damage variable $\zeta$. 
%Together with linear and box constraints \eqref{linear_constraints_QP} and \eqref{box_constraints_QP}, the QP has a unique solution. 
Note that the optimal 
%couple 
pair $(\ZZ_+, \ZZ_-)$ must satisfy $\ZZ_+\ZZ_- {=} 0$ in all nodes, i.e.\ both 
variables cannot be positive. This can be easily seen by contradiction: 
If $\ZZ_+  \ZZ_- {>} 0$ in some node, then a different pair 
$(\ZZ_+{-}\min\{\ZZ_+,\ZZ_-\},\ZZ_-{-}\min\{\ZZ_+,\ZZ_-\})$ would again satisfy 
the constraints \eqref{linear_and_box_constraints_QP} but would provide a smaller 
energy value in \eqref{minimization-2-QP}.
%, which would contradict with the minimal energy property of the couple $(\ZZ_+,  \ZZ_-)$.

Our MATLAB implementation is available for download at Matlab Central as a 
package {\it Continuum undergoing combined elasto-plasto-damage transformation},
cf.\ \cite{code}.
%under {\small\url{http://www.mathworks.com/matlabcentral/fileexchange/authors/37756}}. 
It is based on an original elastoplasticity code related to multi-surface models \cite{BrCaVa05QSBVPMS}. 
The code is simplified to work with one surface variable only (which corresponds to the 
classical model of kinematic hardening) and sets the hardening parameter to zero to enforce 
perfect plasticity. It partially utilizes vectorization techniques of \cite{RaVa13FMAFEM2D3D} 
and works reasonably fast also for finer rectangular meshes.

\section{Illustrative computational simulations}\label{sec-simul}
%        ~~~~~~~~~~~~~~~~~~~~~~~~~~~~~~~~~~~~~~
We consider a time-simulation of a 2-dimensional continuum visualized 
in Figure~\ref{fig5-plast-dam-geom} describing two ``plates'' moving 
horizontally in opposite directions with the constant velocity 
$\pm10^{-8}$m/s\,$\doteq$\,30\,cm/yr. 
%31536000=365x24x60x60
The model 
%from Example~\ref{exa6:alternative-plast} 
has 
%(in some variants)
applications in geophysics, specifically in modelling of tectonic and seismic
processes in crustal parts of the earth lithosphere in the relatively 
short or very short time scales (meaning substantially less than a million 
of years) where the small-strain concept and solid mechanics are well relevant.
The hardening is naturally considered zero.
%This requires either viscosity acting on $\DT e_\mathrm{el}$ as in 
%Remark~\ref{rem6:alternative-plast}, cf.\ \cite{RouVal??},
%%\ref{rem6:alternative-thermoplast} on 
%%p.\,\pageref{rem6:alternative-thermoplast} below, 
%or higher-order terms in $\calV$, cf.\ \cite{RoSoVo13MBMF}.
The damage variable is in the position of a so-called ageing. 
%A quasistatic
%variant, i.e.\ $\varrho=0$ and $\mathbb D=0$ in \eq{eq6:plast-dam}, works too,
%having the structure of \eq{GM+++} with $M=0$. 
The healing together with the damage-dependent 
plastic yield stress allow for periodically alternating fast damage and 
slow healing under external loading with constant velocity, which is a 
typical stick/slip-type events of flat partly damaged subdomains (so-called
lithospheric faults) manifested by re-occurring earthquakes. 

%\hrule\hrule\hrule

\begin{figure}[th]
\psfrag{400}{\scriptsize $400\,$m}
\psfrag{100}{\footnotesize $100\,$m}
\psfrag{20}{\footnotesize $20\,$m}
\psfrag{8}{\footnotesize $8\,$m}
\psfrag{W}{\small $\Omega$}
\psfrag{partly damaged}{\footnotesize
\begin{minipage}[t]{10em}
initially partly\\[-.2em]
damaged zone\end{minipage}}
\psfrag{GN}{\footnotesize $\GN$}
\psfrag{GD}{\footnotesize $\GD$}
\begin{center}
\hspace*{.5em}\includegraphics[width=.7\textwidth]{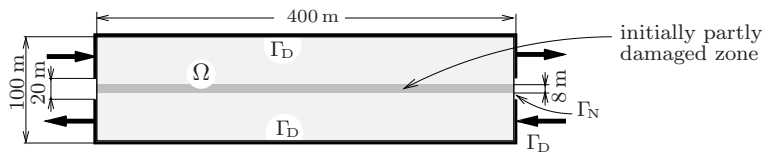}
\end{center}
\vspace*{-.9em}
\caption{Geometry used for the computational experiment, imitating the fault 
between two plates moving horizontally in opposite directions. The time-dependent 
Dirichlet conditions are prescribed on $\GD$, 
%moving horizontally in opposite directions with 
using the constant velocity $\pm10^{-8}{\rm m/s}\,\doteq\,30\,{\rm cm/year}$.
}
\label{fig5-plast-dam-geom}
\end{figure}

The domain $\Omega$ is assumed to be occupied by an elastic continuum specified 
by an isotropic homogeneous elasticity tensor in the form \eqref{C-ass-implem}
with 
%$\mathbb{C}(1)$ (i.e.\ the non-damaged 
%%(or healthy) 
%material) using 
$\lambda_1 = 7.5\,$GPa and $\mu_1=11.25\,$GPa
(which corresponds to Young's modulus $E_{_\mathrm{Young}}\!\!\!\!=27\,$GPa and Poisson' ratio $\nu=0.2$
in the non-damage state)
% (it corresponds to Lam\'e parameters $\lambda = 7.5\,$GPa, $\mu=11.25\,$GPa. 
while the damaged material uses ten-times less moduli, 
i.e.\ $\lambda_0= 0.75\,$GPa and $\mu_0=1.125\,$GPa in \eqref{C-ass-implem}.
%The elasticity tensor decreases in the completely damage configuration to $\mathbb{C}(0)=\mathbb{C}(1)/10$. 
The yield stress $\sigma_{\rm y}$ in \eqref{special-yield-stress} 
%lies 
ranges between the values $\sigma_{\mbox{\tiny\rm Y},1}
%\sigma_{\rm y} (1)
=2\,$MPa and $\sigma_{\mbox{\tiny\rm Y},0}=\sigma_{\mbox{\tiny\rm Y},1}\times10^{-12}$.
%\sigma_{\rm y} (0)=\sigma_{\rm y} (1)/10^{12}\,$MPa. 
The damage-dissipation potential \eqref{damage_disipation_potential} is specified by constants 
$a_1=100\,$GPa\,s
%,  $a_2=10\,$MPa\,s, 
and $a_3=10\,$Pa while 
the damage viscosity $a_2$ will vary. The stored energy of damage is 
%assumes the constant 
$b_1=0.001\,$J/m$^{3}$ 
with 
%and the stored energy of gradient damage 
the damage length-scale coefficient 
$\kappa=0.001\,$J/m. The initial conditions ensure that 
$\pi_0\!=\!0$, $\zeta_0\!=\!1$ (or $\zeta_0\!=\!1/2$ in a middle narrow horizontal stripe). 
%It all numerical tests that 

The first numerical test is run for discrete times in the interval $0\le t\le 400\,$ks 
with 
%the initial time $t_0=0\,$ks and 
the equidistant time partition using the time-step $\tau=1\,$ks. The spatial discretisation
of the domain $\Omega$ used a uniform triangular mesh with $4608$ elements and 2373 nodes; 
this mesh is available by setting 'level=2' in the code \cite{code}, while finer uniform 
meshes can be generated by putting higher values of  the `level' parameter. Thus, 400 
time-steps are computed and Figure \ref{fig:fields} displays space-distributions of the 
shifted damage $1-\zeta$, of the Frobenius norm of the plastic strain $\pi$, and of the 
von Mises stress $|\dev(\sigma)|$ at %some discrete times. 
selected instants.

\begin{figure}[h]%\label{fig:fields}
\center
\quad\includegraphics[width=0.84\textwidth]{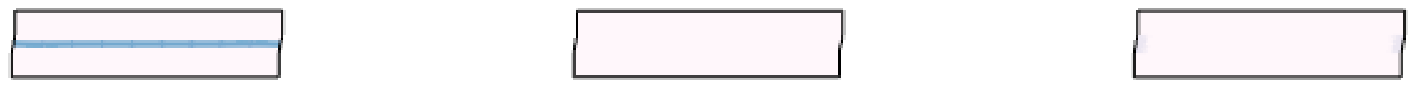}\quad t=20\,ks
\\[.3em]
\quad\includegraphics[width=0.84\textwidth]{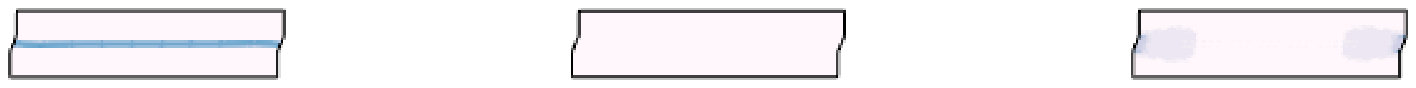}\quad t=40\,ks
\\[.3em]
\quad\includegraphics[width=0.84\textwidth]{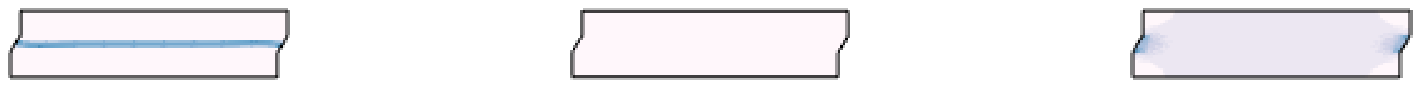}\quad t=60\,ks
\\[.3em]
\quad\includegraphics[width=0.84\textwidth]{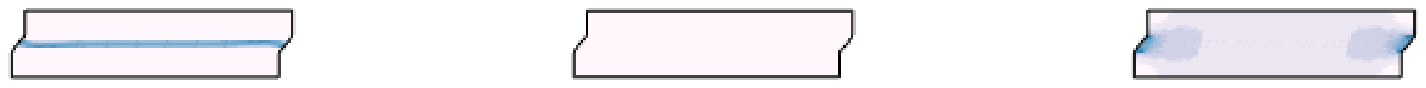}\quad t=80\,ks
\\[.3em]
\quad\includegraphics[width=0.84\textwidth]{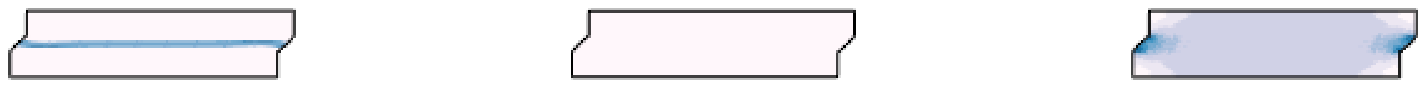}\quad  t=100\,ks
\\[.3em]
\quad\includegraphics[width=0.84\textwidth]{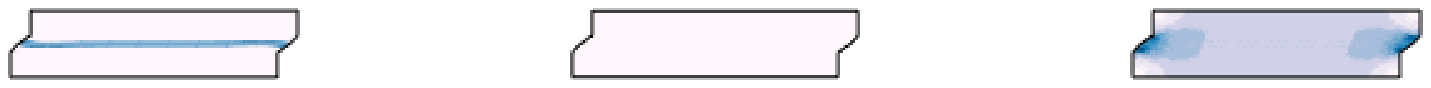}\quad  t=120\,ks
\\[.3em]
\quad\includegraphics[width=0.84\textwidth]{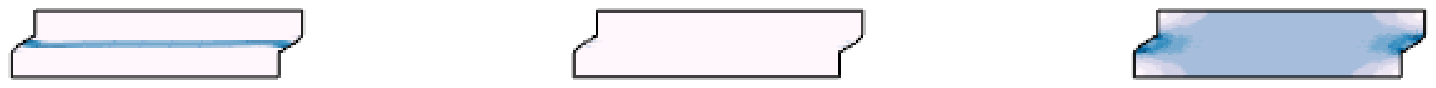}\quad t=140\,ks
\\[.3em]
\quad\includegraphics[width=0.84\textwidth]{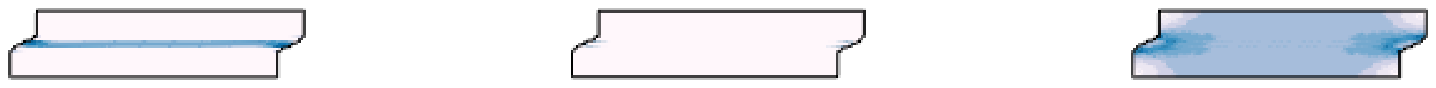}\quad  t=160\,ks
\\[.3em]
\quad\includegraphics[width=0.84\textwidth]{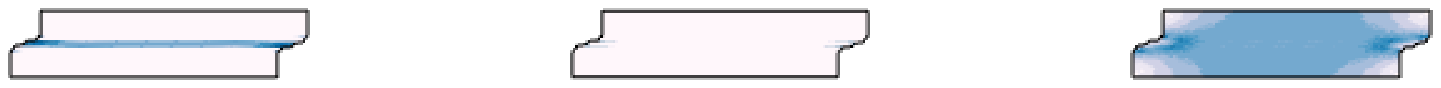}\quad  t=180\,ks
\\[.3em]
\quad\includegraphics[width=0.84\textwidth]{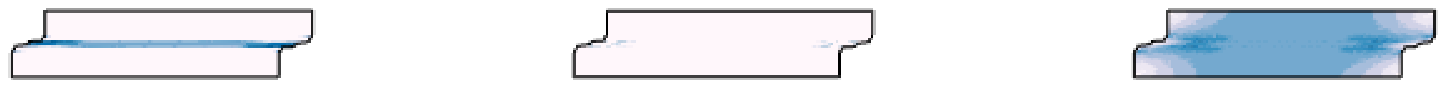}\quad  t=200\,ks
\\[.3em]
\quad\includegraphics[width=0.84\textwidth]{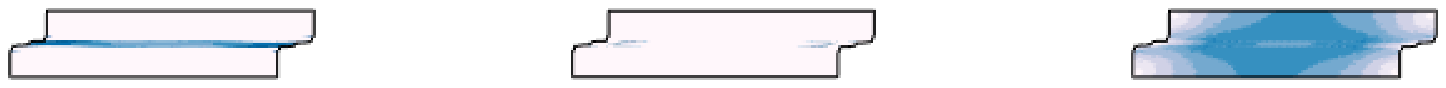}\quad  t=220\,ks
\\[.3em]
\quad\includegraphics[width=0.84\textwidth]{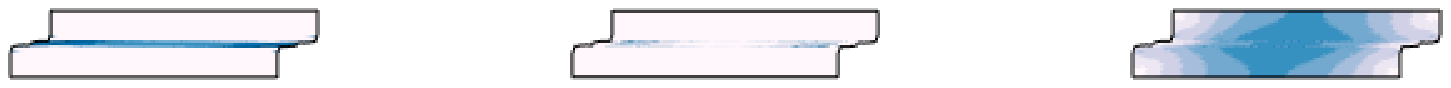}\quad  t=240\,ks
\\[.3em]
\quad\includegraphics[width=0.84\textwidth]{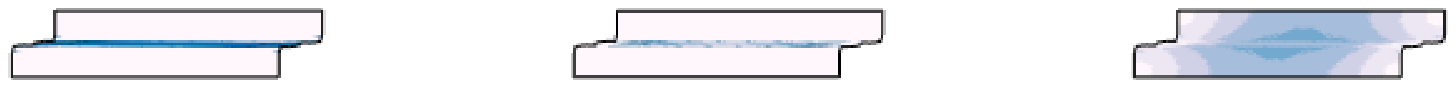}\quad t=260\,ks
\\[.3em]
\quad\includegraphics[width=0.84\textwidth]{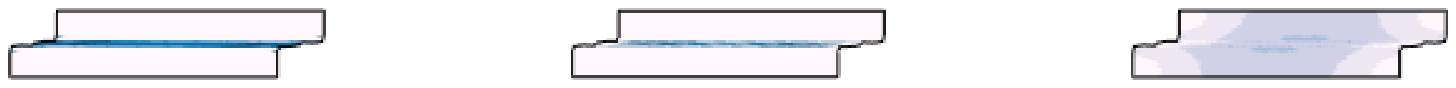}\quad t=280\,ks
\\[.3em]
\quad\includegraphics[width=0.84\textwidth]{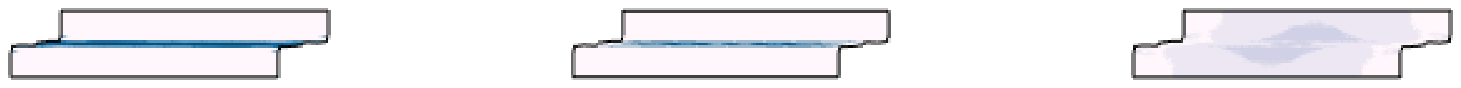}\quad t=300\,ks
\\[.3em]
\quad\includegraphics[width=0.84\textwidth]{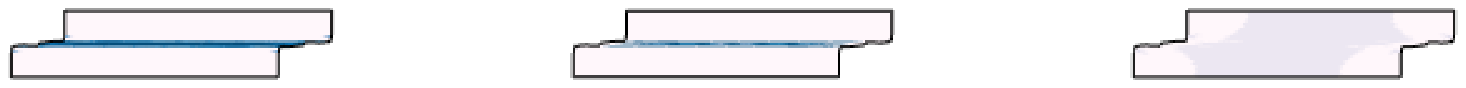}\quad t=320\,ks
\\[.1em]
\hspace*{-1em}\includegraphics[width=0.89\textwidth]{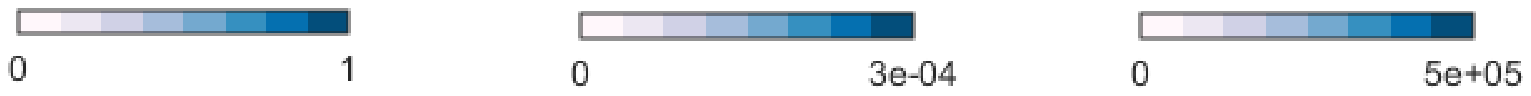}\qquad\qquad
%\hspace{0.09\textwidth}
%\vspace{-0.3em}
\caption{Evolution of space-distributions of damage
%the shifted damage $1-\zeta$ 
(the left column, displaying $1{-}\zeta$), of the 
%(Frobenius norm of the) 
plastic strain (the middle column, displaying the Frobenius norm $|\pi|$) and of the von Mieses stress 
(the right column, displaying $|\dev(\sigma)|$). 
%are displayed for uniformly growing discrete times $t$. 
The displacement of the deformed domain is displayed magnified by the factor 12500. 
Distributions were computed for 
% the rate-dependent damage parameter
damage viscosity $a_2=10\,{\rm MPa\,s}$.  }
\label{fig:fields}\end{figure}

In order to see how the quality of discrete solutions depends on the time-step $\tau$, 
similar numerical tests are run for two additional time-steps $\tau=5\,$ks  and $\tau=10\,$ks. 
%resulting in 80 and 40 time-steps respectively. 
The resulting energy balance \eqref{engr-bal-disc} is displayed in 
%the right-hand part of 
Figure \ref{fig:upper_bounds}. Naturally, it is best fulfilled for the smallest 
considered time-step  $\tau=1\,$ks. 
%\CHECK{
%The right-hand part of 
Figure \ref{fig:stress-strains} 
shows the (horizontal component of the) reaction force
%average von Mises stresses over discrete times.} The 
which is here evaluated (very roughly) as an average 
%here is evaluated only 
from element values of von Mises stresses in the middle 
narrow horizontal stripe (i.e.\ the fault zone) shown in Figure \ref{fig5-plast-dam-geom}. 
A comparison of Figures \ref{fig:upper_bounds} and \ref{fig:stress-strains} indicates 
that the energy balance \eqref{engr-bal-disc} is better satisfies in the purely 
elasto-plastic regime than within the undergoing damage. 
%In the combined elasto-plasto-damage regime, smaller time steps are needed to obtain 
%a sharp energy balance. 
This becomes even more apparent if the damage process is speeded up by setting a 
smaller value $a_2=0.1\,$MPa\,s, cf.\ the left-hand parts of Figures \ref{fig:upper_bounds} 
and \ref{fig:stress-strains} versus the right-hand parts.
%. The left-hand parts of Figures \ref{fig:upper_bounds} 
%and \ref{fig:stress-strains} show the corresponding graphs. 
%The averaged values of 
%von Mises stresses are converging (only the curve for the largest considered time-step 
%$\tau=10\,$ks is qualitatively different), while the energy balance shows worse convergence. 
%This can be expained that the faster damage (caused by the smaller value of $a_2$) would 
%be better captured on even smaller time scales. 
%It is expected that a convenient time-adaptive strategy based on the energy 
%balance \eqref{engr-bal-disc} would recover the energy balance. 

\begin{figure}
\center
%\psfrag{par}{\LARGE $\bigg\}$}
\psfrag{Diss}{\scriptsize\begin{minipage}[t]{10em}\hspace*{.2em}work of\\[-.2em]external\\[-.2em]\hspace*{.7em}load\end{minipage}}
\psfrag{R}{\scriptsize\begin{minipage}[t]{10em}\hspace*{.8em}dissipated\\[-.2em]
\hspace*{1em}+\,stored\\[-.2em]\hspace*{1.5em}energy\end{minipage}}
\includegraphics[width=0.42\textwidth]{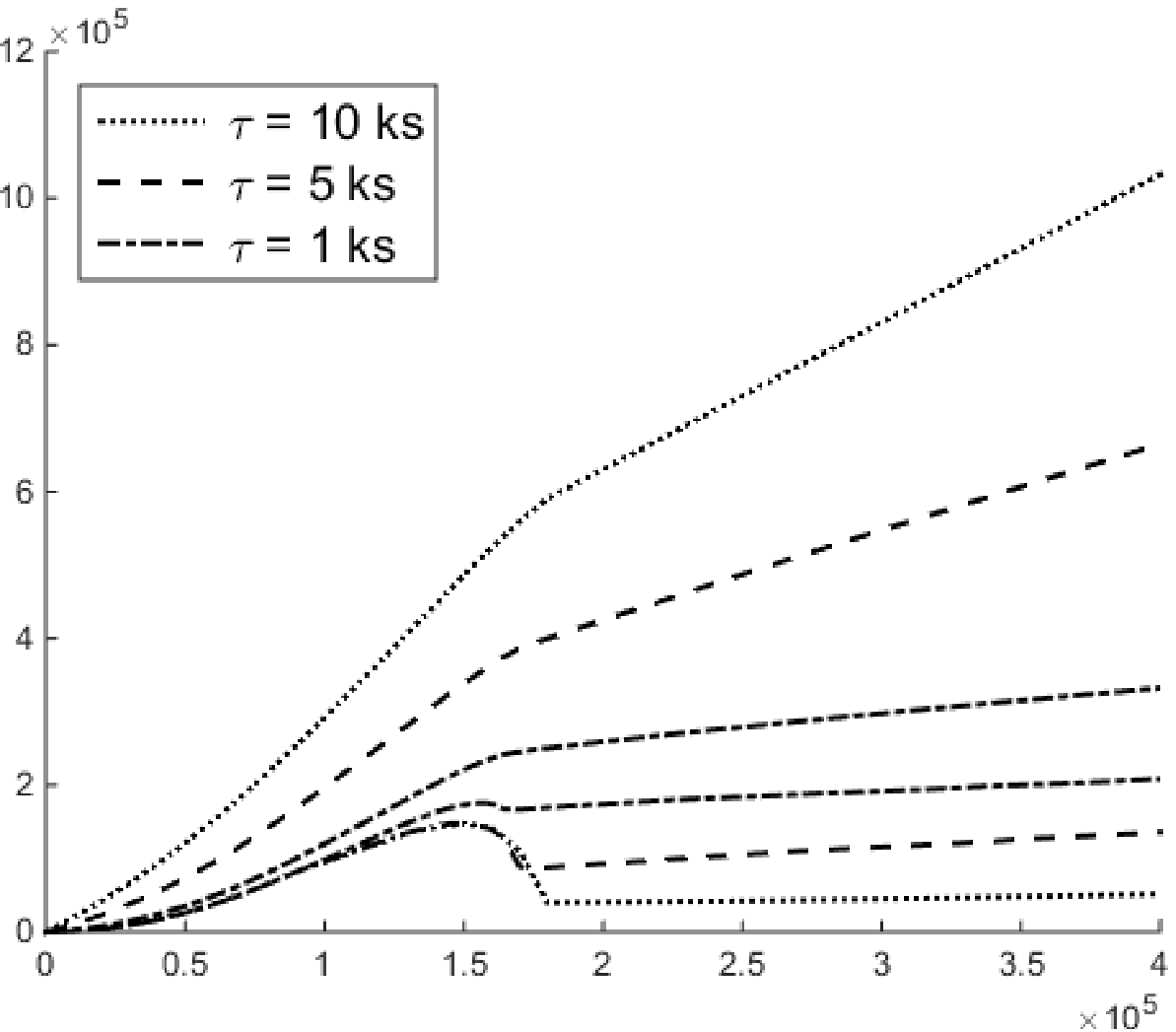}
\hspace*{-.1em}\includegraphics[width=.05\textwidth,height=.315\textwidth]{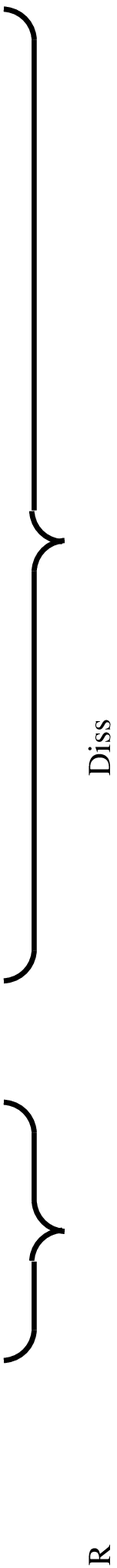}
\hspace{.05\textwidth}
\includegraphics[width=0.42\textwidth]{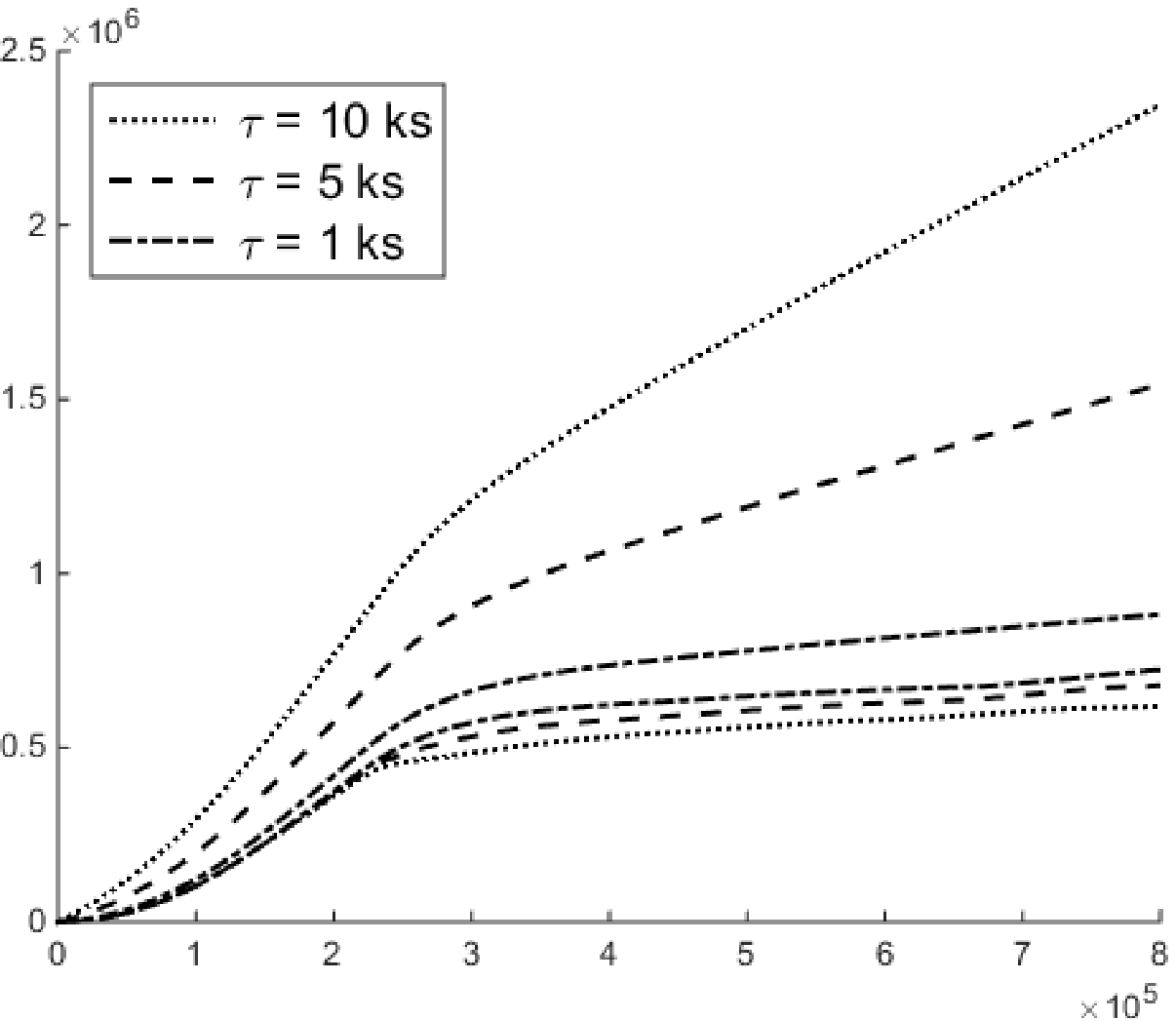}
\vspace*{-1em}
\caption{Evolution of the stored and dissipated energy (=\,the left-hand side
of \eqref{engr-bal-disc} for $T$ varying) and the work of external loading 
(=\,the right-hand side of \eqref{engr-bal-disc} for $T$ as a current time $t$)
calculated for three different values of the time steps $\tau=10,\ 5,\ 1\,{\rm ks}$, 
documenting the convergence of \eqref{engr-bal-disc} towards the
energy equality \eqref{engr-bal} proved in Proposition~\ref{prop-conv}.
For less viscous damage this convergence is naturally slower than
for a more viscous damage, cf.\ 
the left figure for $a_2=0.1\,{\rm MPa\,s}$ vs 
the right one for $a_2=10\,{\rm MPa\,s}$.
%Other two damage-dissipation constants are identical: rate-dependent healing parameter $a_1=10^{11}$, rate-independent damage parameter $a_3=10$ for %all curves.
}\label{fig:upper_bounds}
\end{figure}

\begin{figure}
\center
\includegraphics[width=0.42\textwidth]{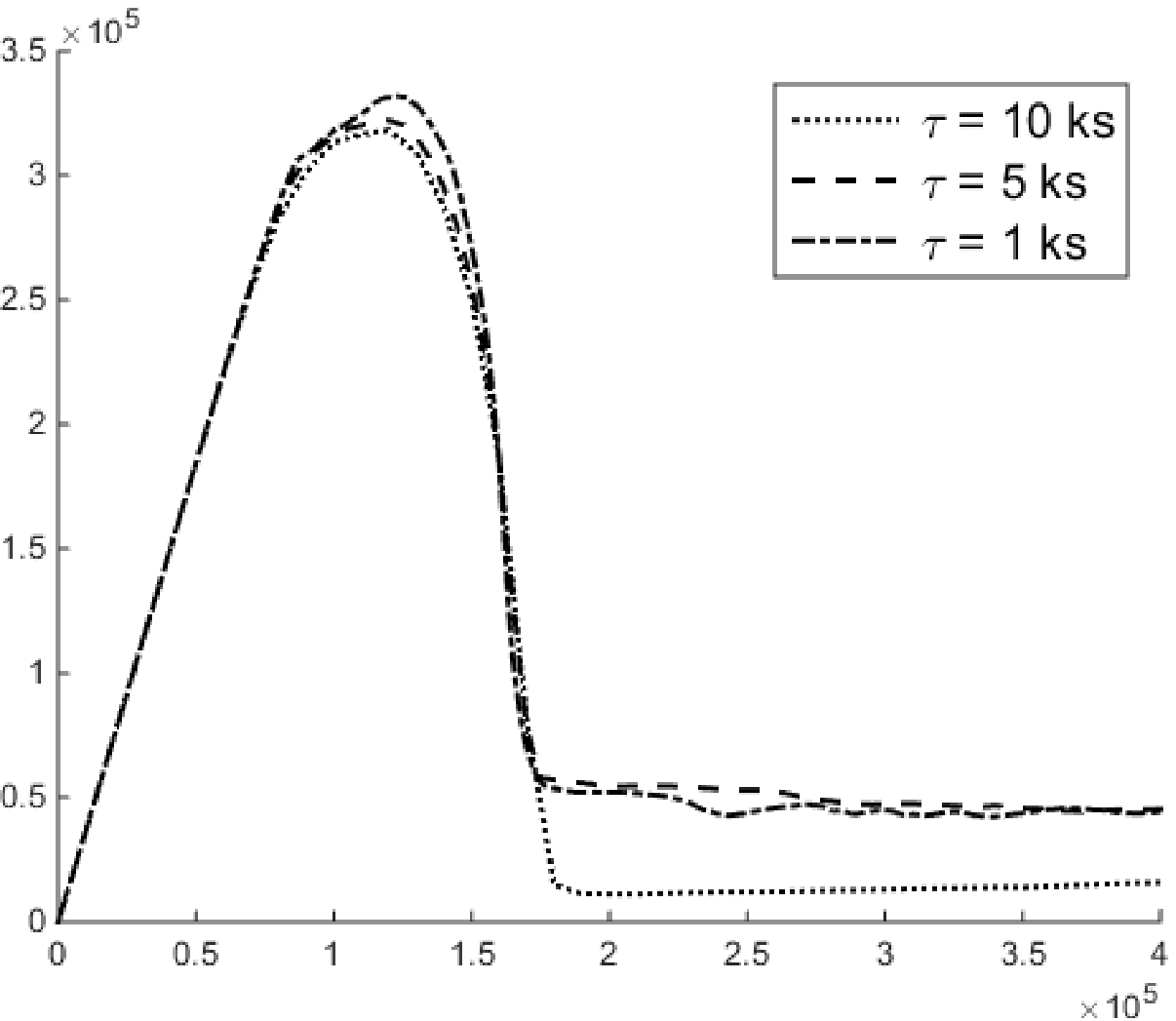}
\hspace{.1\textwidth}
\includegraphics[width=0.42\textwidth]{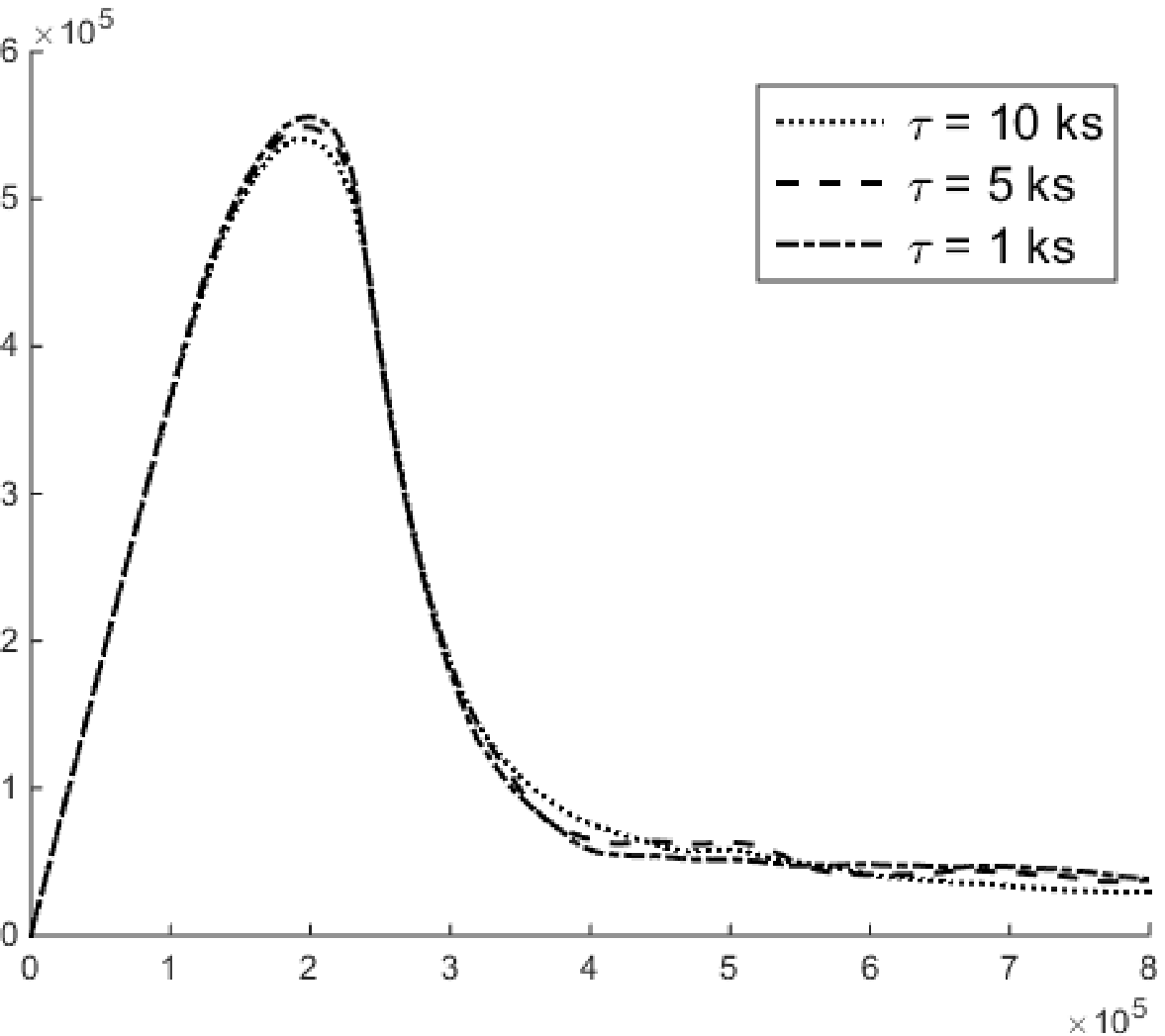}
\vspace*{-1em}
\caption{Evolution of the reaction force 
%(spatially averaged) von-Mises stress 
corresponding to Figure~\ref{fig:upper_bounds}; the time scales on 
the left and the right figures are different. Noteworthy, the force 
response is well converged even in situations when the energetics 
on Figure~\ref{fig:upper_bounds} exhibits still big gaps.
%Corresponding averaged stress - time curves for two different rate-dependent damage parameters:  $a_2=0.1\,{\rm MPa\,s}$ (left) and 
%$a_2=10\,{\rm MPa\,s}$ (right) from Figure \ref{fig:upper_bounds}.
%
%Other two damage-dissipation constants are identical: rate-dependent healing parameter $a_1=10^{11}$, rate-independent damage parameter $a_3=10$ for all curves.
}\label{fig:stress-strains}
\end{figure}

%Average stress - time curves
Dependence of the reaction-force evolution for varying viscosity of damage 
%for the smallest time steps are
is shown in Figure~\ref{fig:stress-strains_comparison} for $a_2$ as in 
Figures~\ref{fig:upper_bounds}--\ref{fig:stress-strains} compared also with
a smaller viscosity 
%including a new curve for another small parameter 
$a_2=1\,$kPa\,s
%. Smaller values such that $a_2=0.01\,$kPa\,s,
 which 
already provides
%a curves 
a response essentially identical to the 
%$a_2=1\,$kPa\,s curve.
even smaller viscosity 
%$a_2=1\,$kPa\,s.
$a_2=0.01\,$kPa\,s (not displayed in Figure~\ref{fig:stress-strains_comparison}) 
where conservation of energy is numerically still more difficult to achieve.
This indicates a certain tendency for convergence towards the model
using rate-independent damage combined with rate-dependent healing (as in 
\cite[Sect.\,5.2.7]{MieRou15RIPT}) and with perfect plasticity, which is 
theoretically not justified, however.

\begin{figure}\center
\includegraphics[width=0.6\textwidth,height=0.35\textwidth]{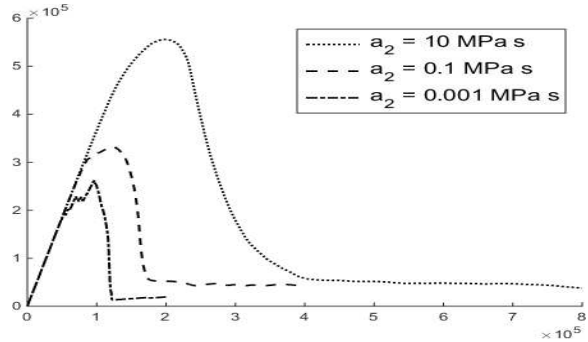}
\vspace*{-1em}
\caption{Dependence of the 
%stress/strain
repulsive-force response on the viscosity of
damage, the cases $a_2=10$ and $0.1\,{\rm MPa\,s}$
are (parts of) Figure~\ref{fig:stress-strains} and are here compared 
also with even less viscous damage for $a_2=1\,{\rm kPa\,s}$
which gives essentially the same response as for the nearly
inviscid case $a_2=0.01\,{\rm kPa\,s}$ (not displayed, however);
the time-step $\tau=1\,{\rm ks}$. For decreasing viscosity, 
the rupture occurs earlier and propagates faster,
showing a tendency to converge to an inviscid rate-independent (and
theoretically not justified) damage model.
%
%Averaged stress - time curves compared for rate-dependent damage parameters: $a_2 \in \{ 10, 0.1, 0.01 \}\,{\rm MPa\,s}$ with  all curved computed for the time-step $\tau=1\,{\rm ks}$.
%
% Other two damage-dissipation constants: rate-dependent healing parameter $a_1=1e11$, rate-independent damage parameter $a_3=1e1.$
}\label{fig:stress-strains_comparison}
\end{figure}

%\begin{remark}[\emph{Perfect plasticity, concept of ageing}]
%\upshape

%The perfect plasticity allows for modelling of a very narrow core of fault
%(sometimes of the width much below a meter). In fact, 
%an slip surfaces with zero width can develop by this perfect plasticity
%model. In contrast to it,
%the coefficient $\kappa>0$ controls the typical width of the damage zone
%surrounding the core of fault (having the width of tens or hundreds of meters),
%cf.\ e.g.\ \cite{LyBZAg97DDFF}.
%In fact, a (small) hardening mentioned in Section~\ref{sec-comp} \COMMENT{???}
%prevents these sharp slip bands and, in
%geophysical applications we have in mind, allows to control 
%the (presumably narrow) width of fault cores similarly as $\kappa$
%controls the (presumably) wider of the damage zone surrounding these cores.
%Our perfect-plasticity model can then be understood as an asymptotic
%study for vanishing hardening, cf.\ also \cite{BaMiRo12QSSP}.
%In contrast to a plastic-strain evolution governed 
%directly by the damage used often in geophysical models without
%a rational thermodynamical justification, cf.\ e.g.\ 
%\cite[Formula (7)]{LyBZAg97DDFF} or ..............................
%we use here the 
%rather conventional plasticity allowing, in particular, for 
%specifying (and a-posteriori checking of) an explicit energy balance
%and proving numerical stability and convergence of numerical 
%discretisations.

Let us eventually remark that the a-posteriori information obtained 
from the residuum in the discrete energy balance 
\eqref{engr-bal-disc} written at a current time $t$ (as also used in 
Figure~\ref{fig:upper_bounds}) can be used to control adaptively 
the time step in a way to keep the numerical error in the energy
under an a-priori prescribed tolerance and, on the other hand, 
not to waste computational time by making too small time steps 
in periods of slow evolution. We intentionally presented our numerical
simulation on equidistant time partitions, but for actual 
geophysical simulations with very big difference in time scale 
between fast damage (earthquakes) and very slow healing, such 
an adaptivity is necessary.

{\small
\section*{Acknowledgments}
This research has been supported by GA\,\v CR through the project
%s 201/12/0671 
%``Variational and numerical analysis in nonsmooth continuum mechanics'',
13-18652S ``Computational modeling of damage and transport processes 
in quasi-brittle materials''
%, and 201/10/0357 ``Modern mathematical and computational 
%models for inelastic processes in solids'' 
and 14-15264S 
``Experimentally justified multiscale modelling of shape memory alloys'',
with also the also institutional support RVO:61388998 (\v CR).
%\COMMENT{SOME OTHER GRANTS??}
}

%.... we neglect inertial (and thus the possible elastic waves), ...

%\INSERT{A COMPARISON WITH } \cite{AlMaVi14GDMC,Cris??GSQE}

%\INSERT{A COMPARISON WITH C.Heinemann, C.Kraus:
%Existence of weak solutions for a PDE system
%describing phase separation and damage 
%process including inertial effects.
%Disc. Cont. Dynam. Syst. A 35 (2015), 2565-2590.}

\newpage

\bibliographystyle{abbrv}
%\bibliography{Bibliography}
\bibliography{tr-jv-plast-dam}
\end{document}